\documentclass[12pt]{amsart} 
\usepackage{verbatim, latexsym, amssymb, amsmath,color, mathrsfs}
\usepackage{enumitem}
\usepackage{epsfig}
\usepackage{xcolor}

\usepackage{stmaryrd}

\makeatletter
\renewcommand{\tocsection}[3]{%
  \indentlabel{\@ifnotempty{#2}{\bfseries\ignorespaces#1 #2\quad}}\bfseries#3}
\renewcommand{\tocsubsection}[3]{%
  \indentlabel{\@ifnotempty{#2}{\ignorespaces#1 #2\quad}}#3}

\newcommand\@dotsep{4.5}
\def\@tocline#1#2#3#4#5#6#7{\relax
  \ifnum #1>\c@tocdepth 
  \else
    \par \addpenalty\@secpenalty\addvspace{#2}%
    \begingroup \hyphenpenalty\@M
    \@ifempty{#4}{%
      \@tempdima\csname r@tocindent\number#1\endcsname\relax
    }{%
      \@tempdima#4\relax
    }%
    \parindent\z@ \leftskip#3\relax \advance\leftskip\@tempdima\relax
    \rightskip\@pnumwidth plus1em \parfillskip-\@pnumwidth
    #5\leavevmode\hskip-\@tempdima{#6}\nobreak
    \leaders\hbox{$\m@th\mkern \@dotsep mu\hbox{.}\mkern \@dotsep mu$}\hfill
    \nobreak
    \hbox to\@pnumwidth{\@tocpagenum{\ifnum#1=1\bfseries\fi#7}}\par
    \nobreak
    \endgroup
  \fi}
\AtBeginDocument{%
\expandafter\renewcommand\csname r@tocindent0\endcsname{0pt}
}
\def\l@subsection{\@tocline{2}{0pt}{2.5pc}{5pc}{}}
\makeatother

\usepackage{geometry}
\usepackage{hyperref}

\DeclareMathOperator{\End}{End}

\DeclareMathOperator{\spt}{spt}
\DeclareMathOperator{\Id}{Id}
\DeclareMathOperator{\diam}{diam}

\DeclareMathOperator{\Vol}{Vol}
\DeclareMathOperator{\length}{length}
\DeclareMathOperator{\Area}{Area}

\DeclareMathOperator{\dist}{dist}
\DeclareMathOperator{\injrad}{injrad}

\DeclareMathOperator{\Lip}{{Lip}}
\DeclareMathOperator{\Jac}{Jac}
\DeclareMathOperator{\spherevol}{SphereVol}

\DeclareMathOperator{\id}{id}
\DeclareMathOperator{\set}{set}
\DeclareMathOperator{\tr}{Tr}

\DeclareMathOperator{\dvol}{dvol}

\newtheorem{theo}{Theorem}[section]

\newtheorem{lemme}[theo]{Lemma}
\newtheorem{definition}[theo]{Definition}
\newtheorem{coro}[theo]{Corollary}

\theoremstyle{definition}
\newtheorem{remarque}[theo]{Remark}

\newtheorem{question}{Question}

\begin{document}
\title[Spherical volume and spherical Plateau problem]
{Spherical volume and spherical Plateau problem}

\author{Antoine Song}
\address{California Institute of Technology\\ 177 Linde Hall, \#1200 E. California Blvd., Pasadena, CA 91125}
\email{aysong@caltech.edu}

\maketitle

\begin{abstract} 
Given a closed oriented manifold or more generally a group homology class, we introduce the spherical Plateau problem, which is a variational problem corresponding to a topological invariant called the spherical volume.
In principle, its solutions should be realized by minimal surfaces in quotients of spheres. 
We explain that in many geometrically interesting cases, those solutions are essentially unique.  
We start with a  review of the Ambrosio-Kirchheim theory of metric currents, and the barycenter map method developed by Besson-Courtois-Gallot. 
Then, we outline the following applications:
\begin{enumerate}
\item the intrinsic uniqueness of spherical Plateau solutions for negatively curved, locally symmetric, closed oriented manifolds,
\item  the intrinsic uniqueness of spherical Plateau solutions for all 3-dimensional closed oriented manifolds, 
\item the construction of higher-dimensional analogues of hyperbolic Dehn fillings. 
\end{enumerate}
We also propose some open questions. 

\end{abstract}


\tableofcontents

\section*{Introduction}
The classical \emph{Plateau problem}, a fundamental question in Differential Geometry, concerns the existence and regularity of surfaces of least area spanning a given boundary contour inside the $3$-dimensional Euclidean space. The term ``Plateau problem'' has been extended to encompass any situation where the objective is to construct and study ``\emph{minimal surfaces}'', which are minimizers of a volume or area functional subject to  topological or geometric constraints.
 As a concrete example, consider a bounded Riemannian manifold $(N,g)$ of finite or infinite dimension, and fix an $n$-dimensional integer homology class $h\in H_n(N;\mathbb{Z})$ where $n$ is a positive integer. Roughly speaking, the ``volume'' or ``area'' of this homology class is defined as follows:
$$\Area(h):= \inf\{\Area(C);\quad \text{$C$ is a cycle in $N$ representing $h$}\}.$$
Here cycles can be intuitively understood as generalized $n$-submanifolds of $N$.
Let $C_i$ be a sequence of cycles  in $N$ representing $h$, which is minimizing in the sense that
$$\lim_{i\to \infty} \Area(C_i) = \Area(h).$$
Thanks to adequate compactness results, a subsequence $C_{i_j}$ converges to some limit space $C_\infty$ called \emph{Plateau solution}. 
The properties of this limit space, such as regularity, uniqueness, and geometric structure, are the main focus of the Plateau problem. In some way, this space $C_\infty$ can be considered as an ``optimal  geometric representative'' of the homology class $h$ of $(N,g)$. Uniqueness of Plateau solutions holds only in exceptional situations, and this paper is about a natural Plateau problem in infinite dimension, for which interesting uniqueness results hold or are conjectured to be true.

A central question in the study of  Plateau problems is to define the concepts of ``cycle'', ``area'', ``volume'' and ``convergence''. We will rely on  the framework of \emph{integral currents} in Geometric Measure Theory, as it  provides the most far-reaching existence and regularity results  so far, though there are other possible choices. In this context, the Plateau problem has been thoroughly studied for finite dimensional manifolds $N$:  in particular 
if the Riemannian manifold $(N,g)$ is finite dimensional and closed, any $k$-dimensional integer homology class $h$ of $N$ admits an area-minimizing integral current representative $C_\infty$ in $N$ with area equal to $\Area(h)$, which is smooth outside of a codimension two subset (this object is a ``generalized $k$-dimensional minimal surface'').
This major result follows from the successive works of De Giorgi, Federer-Fleming, Allard, Almgren, and many others. See \cite[Theorem 9.6]{FF60} for the existence statement and the surveys \cite{Ambrosio16,DeLellis16} for the partial regularity problem.

However when the underlying manifold $N$ is infinite-dimensional and hence not locally compact, the situation becomes more challenging. Efforts have been made to extend the theory of currents to arbitrary complete metric spaces, which started with Ambrosio-Kirchheim's article \cite{AK00}. This theory has since been applied, extended, and revisited by many researchers, including Lang \cite{Lang11}, Ambrosio-Schmidt \cite{AS13}, Wenger \cite{Wenger07,Wenger11,Wenger14}, Sormani-Wenger \cite{SW11}, Ambrosio-De Lellis-Schmidt \cite{ADLS18} etc. It is this extended  theory which allows us  to define and construct Plateau solutions in general. More concretely, if $(N,g)$ is an infinite-dimensional complete  Riemannian manifold with finite diameter and $h\in H_n(N;\mathbb{Z})$, a cycle in $(N,g)$ is by definition a boundaryless  integral current with compact support in $(N,g)$ in the sense of \cite{AK00}. Its area is given by the notion of mass of an integral current \cite{AK00}. Given a minimizing sequence of cycles $C_i$ representing $h$, by Wenger's compactness theorem \cite{Wenger11}, $C_i$ subsequentially converges in the intrinsic flat topology to an integral  
current space $C_\infty$ \cite{SW11}. An integral current space is roughly speaking determined by an underlying metric space $(X,d)$ and an integral current $S$ in the completion of $(X,d)$. Any such $C_\infty$ is called a Plateau solution for $h$.

Inspired by earlier works of Besson-Courtois-Gallot \cite{BCG91,BCG95, BCG96}, we consider the \emph{spherical Plateau problem}, a natural infinite-dimensional Plateau problem at the interface of geometric measure theory, geometric group theory and topology. The ambient manifold $(N,g)$ is of the following form: 
$$(N,g)=( {S^\infty}/\lambda_\Gamma(\Gamma), \mathbf{g}_{\mathrm{Hil}})$$ where $\Gamma$ is a countable group, $ {S^\infty}$ is the unit sphere in $\ell^2(\Gamma)$, $\Gamma$ acts on $ {S^\infty}$ by the left regular representation 
$\lambda_\Gamma:\Gamma\to \End(\ell^2(\Gamma))$ and $ {S^\infty}/\lambda_\Gamma(\Gamma)$ is endowed with the quotient of the round metric $\mathbf{g}_{\mathrm{Hil}}$. Given $h\in H_n( {S^\infty}/\lambda_\Gamma(\Gamma);\mathbb{Z})$, the invariant $\Area(h)$ is called the spherical volume of $h$ and denoted by
$$\spherevol(h):= \inf\{\Area(C);\quad \text{$C$ is a cycle in $ {S^\infty}/\lambda_\Gamma(\Gamma)$ representing $h$}\}.$$
Versions of that quantity were first defined by Besson-Courtois-Gallot \cite{BCG91,BCG95, BCG96}. 
This invariant is closely related to many other well-known invariants: there is a chain of inequalities connecting  simplicial volume,  spherical volume, minimal volume entropy and minimal volume for the Ricci curvature, as explained in the survey of Kotschick \cite[Section 2]{Kot11}. A more detailed definition is given in Section \ref{definition spp}.
The integral current spaces $C_\infty$ corresponding to $h$, obtained by the recipe described in the previous paragraph, are called \emph{spherical Plateau solutions} for $h$.

\textbf{Aim.} The purpose of this note is to introduce the spherical Plateau problem, and show that for many  special choices of $\Gamma$ and $h\in H_n( {S^\infty}/\lambda_\Gamma(\Gamma);\mathbb{Z})$, the spherical Plateau solutions are in a sense unique and geometrically meaningful. 
Throughout the text, we propose several questions related to the spherical volume and the spherical Plateau problem.

The style of this article is rather informal, nevertheless we try to provide a rather comprehensive  review of the main technical tools.

We will outline the proofs of the three results stated below\footnote{Most of the content in this survey is based on an earlier unpublished preprint \cite{Antoine22}. Theorem \ref{aladd} is proved in details in \cite{Antoine23b}. See also Cosmin Manea's master's thesis \cite{Manea23}.}, which we view as a proof-of-concept justifying a general  study of the spherical Plateau problem.

\textbf{Locally symmetric manifolds of negative curvature.}
The first result pertains to locally symmetric spaces of rank one. Let $(M,g_0)$ be a closed oriented locally symmetric $n$-manifold with negative curvature, and let $\Gamma$ be its fundamental group. Let $h(g_0)$ be the volume entropy of $(M,g_0)$, whose definition is recalled in (\ref{def entropy}). The quotient space $ {S^\infty}/\lambda_\Gamma(\Gamma)$ is a classifying space for $\Gamma$ and the fundamental class $[M]\in H_n(M;\mathbb{Z})$ determines a unique homology class $h_M\in H_n( {S^\infty}/\lambda_\Gamma(\Gamma); \mathbb{Z}) = H_n(M;\mathbb{Z})$. A key step in the proof of the celebrated volume entropy inequality of Besson-Courtois-Gallot \cite{BCG95} was the computation of $\spherevol(h_M)$ with the \emph{barycenter map} method, see Section \ref{section:uniqueness}. The barycenter map will be a central tool in this paper too.

A closed oriented Riemannian $n$-manifold $(W, g_W)$ admits a natural integral current structure $\llbracket 1_W\rrbracket $ induced by its fundamental class $[W]\in H_n(W;\mathbb{Z}) $.
A spherical Plateau solution $C_\infty$ for $h_M$ is called ``intrinsically isomorphic'' to $(W, g_W)$ if the underlying metric space of $C_\infty$
is intrinsically isometric to $(W,g_W)$
via a map sending the current structure of $C_\infty$ to $\llbracket 1_W\rrbracket $, see Definition \ref{definition:intrinsic isomorphism}.

\begin{theo} \label{aladd}
If $(M,g_0)$ is a closed, oriented, locally symmetric manifold of dimension $n\geq 3$, with negative curvature between $-4$ and $-1$, then any spherical Plateau solution for $h_M$ is intrinsically isomorphic to $(M,\frac{h(g_0)^2}{4n} g_0)$.
\end{theo}
This intrinsic uniqueness result leads to the formulation of a new kind of ``area rigidity'' property for the regular representation of $\pi_1(M)$ (see Corollary \ref{rigidity}), and a \emph{Rigidity Conjecture} (see Question \ref{question:unique}).

The proof of Theorem \ref{aladd} has also been applied to study the stability of the entropy inequality of Besson-Courtois-Gallot \cite{BCG95}, see \cite{Antoine23b}. 
Similar questions for higher rank locally symmetric manifolds (even the computation of the spherical volume) are still wide open, because the barycenter map method does not work so well there \cite{CF02}.
Motivated by Theorem \ref{aladd}, it is suggested in  \cite{CN23} to inspect the spherical Plateau problem in the context of Cannon's conjecture.

\textbf{Closed oriented $3$-manifolds.}
Our next result gives another large pool of examples where the spherical Plateau solutions are almost explicit. Let $M$ be a closed oriented $3$-manifold and $\Gamma$ its fundamental group. The fundamental class of $M$ induces a homology class $h_M\in H_3( {S^\infty}/\lambda_\Gamma(\Gamma);\mathbb{Z})$. By the Geometrization theorem \cite{KL08,MT14}, $M$ can be uniquely written as a connected sum $M=M_1\#... \#  M_k$ where each $M_j$ is a closed oriented prime $3$-manifold which is canonically divided into two pieces along some tori: $M_j = M_{j,\mathrm{hyp}} \cup M_{j,\mathrm{Seif}}$ where $M_{j,\mathrm{hyp}}$ carries a finite-volume complete hyperbolic metric, and $ M_{j,\mathrm{Seif}}$  is a union of Seifert manifolds. The disjoint union of the hyperbolic pieces $ M_{j,\mathrm{hyp}}$ endowed with their hyperbolic metrics is denoted $(M_{\mathrm{hyp}}, g_\mathrm{hyp})$.  
\begin{theo}
  If $M$ is closed oriented $3$-manifold with hyperbolic part denoted by $(M_{\mathrm{hyp}},g_\mathrm{hyp})$, then any spherical Plateau solution for $h_M$ is intrinsically isomorphic to $(M_{\mathrm{hyp}},\frac{1}{3} g_\mathrm{hyp})$.
\end{theo}
Hence the hyperbolic part $(M_{\mathrm{hyp}}, g_\mathrm{hyp})$, which is canonically determined by $M$, emerges as the solution of the spherical Plateau problem. The fundamental nature of the hyperbolic part $(M_{\mathrm{hyp}}, g_\mathrm{hyp})$ had previously appeared in other contexts including the Ricci flow, the Yamabe invariant, the simplicial volume, the volume entropy, etc. For instance, the normalized Ricci flow starting at any Riemannian metric on $M$ converges as time goes to infinity to $(M_{\mathrm{hyp}},g_\mathrm{hyp})$ in a multi-pointed Gromov-Hausdorff sense \cite{KL08}. 

\textbf{Plateau Dehn fillings.}
The two previous theorems show that, to some extent, spherical Plateau solutions form a class of spaces which naturally generalizes locally symmetric manifolds of negative curvature. Our last result explores that interpretation further in the context of Dehn fillings. Recall that in dimension $3$, there is a fundamental feature of hyperbolic geometry discovered by Thurston, called hyperbolic Dehn fillings \cite{Thurston97}: in its simplest version, it states that given any finite volume non-compact hyperbolic $3$-manifold $M$, there is a sequence of finite volume hyperbolic manifolds $M_i$ with volume strictly less than that of $M$, and converging geometrically to $M$. In higher dimensions, due to the finiteness theorem of H.C. Wang \cite{Wang72} for locally symmetric negatively curved manifolds, such a phenomenon is impossible in the smooth setting. Nevertheless, as we will see below, by enlarging the set of locally symmetric negatively curved manifolds to the set of spherical Plateau solutions, we do have such an accumulation phenomenon in all dimensions higher than $2$. We start with the higher dimensional $\mathrm{CAT}(0)$ Dehn fillings constructed by Fujiwara-Manning in \cite{FM10}. These fillings are denoted by $M(T_1,...,T_m)$ and are obtained by closing the cusps of a finite volume non-compact hyperbolic $n$-manifold $(M,g_{\mathrm{hyp}})$ with toral cusps. 
 The $T_i$ denote certain subtori in the $m$ cusps of $M$, which are assumed to have injectivity radius greater than $\pi$. 
Each $M(T_1,...,T_m)$ determines a unique homology class $h_{M(T_1,...,T_m)}$ in the corresponding spherical quotient $ {S^\infty}/\lambda_\Gamma(\Gamma)$ where $\Gamma:=\pi_1(M(T_1,...,T_m))$. 
The behavior of Plateau Dehn fillings, namely the spherical Plateau solutions for $h_{M(T_1,...,T_m)}$, is completely analogous to the $3$-dimensional case of hyperbolic Dehn fillings \cite{Gromov81}:
\begin{theo}
We have 
$$\spherevol(h_{M(T_1,...,T_m)}) < \Vol(M,\frac{(n-1)^2}{4n}g_{\mathrm{hyp}}).$$
Moreover for any $\epsilon>0$, if the injectivity radii of $T_1,...,T_m$ are sufficiently large, then 
$$\spherevol(h_{M(T_1,...,T_m)}) > \Vol(M,\frac{(n-1)^2}{4n}g_{\mathrm{hyp}}) - \epsilon$$
and any spherical Plateau solution for $h_{M(T_1,...,T_m)}$ is $\epsilon$-close, in the intrinsic flat topology, to an integral current space intrinsically isomorphic to $(M,\frac{(n-1)^2}{4n}g_{\mathrm{hyp}})$.
\end{theo}
Fujiwara-Manning previously proved that the fillings $M(T_1,...,T_m)$ have simplicial volumes satisfying
$$\|M(T_1,...,T_m)\| \leq \|M\|$$ 
where $\|M\|$ is the simplicial volume of $M$, and conjectured that $\|M(T_1,...,T_m)\|$  converges to $\|M\|$ as the injectivity radius of the tori $T_i$ goes to infinity, but never attains the limit \cite[Conjecture 1.8, Question 1.9]{FM11}. The previous theorem settles  the spherical volume version of this conjecture.
For the construction of Einstein manifolds analogous to Dehn fillings, see \cite{Anderson06a,Bamler12}.

\textbf{Outline of the paper.}
In Section \ref{appendix a}, we give an overview of the theory of integral currents in metric spaces developed by Ambrosio-Kirchheim and others. Results are stated without proofs.

In Section \ref{appendix b}, we explain in details the barycenter map and reproduce the  proofs of some important properties. 
 For clarity, we focus only on the hyperbolic case.

In Section \ref{definition spp}, the spherical Plateau problem is defined, and basic properties of spherical cycles are discussed. 


In Section \ref{section:uniqueness}, 
we sketch Besson-Courtois-Gallot's computation of the spherical volume for hyperbolic manifolds. We briefly explain its close relation to the volume entropy inequality. We then outline the proof of our first uniqueness result, Theorem \ref{uniqueness1}.

In Section \ref{section:uniqueness2}, we indicate how to apply the Geometrization theorem for $3$-manifolds, and extend the arguments of Theorem \ref{uniqueness1} to prove the uniqueness statement for $3$-manifolds, Theorem \ref{uniqueness2}.

In Section \ref{section:uniqueness3},  we describe the idea behind Plateau Dehn fillings and Theorem \ref{dehn fillings}, which can be viewed as an asymptotic uniqueness result.

\subsection*{Acknowledgments} 
I am grateful to G\'{e}rard Besson, Gilles Courtois, John Lott,  Ian Agol, Jason Manning, Camillo De Lellis, Xin Zhou, Hyun Chul Jang, Luca Spolaor, Tamunonye Cheetham-West, Alexander Nolte, Luca Di Cerbo, Richard Bamler, Song Sun, St\'{e}phane Sabourau,  Alexander Nabutovsky,   Ben Lowe, Shi Wang, Cosmin Manea and Zhenhua Liu for many insightful and stimulating discussions during the writing of this article.

A.S. was partially supported by NSF grant DMS-2104254. This research was conducted during the period A.S. served as a Clay Research Fellow.

\section{Preliminaries on metric integral currents} \label{appendix a}

The classical notion of integral currents in finite dimensional manifolds \cite{Federer69} 
exhibits simultaneously several properties explaining the success of the theory: they are mild generalizations of submanifolds, they satisfy strong compactness results and area-minimizers are smooth submanifolds outside of a small singular set \cite{Almgren00,DLS14}. Building on earlier ideas of De Giorgi and Gromov, Ambrosio-Kirchheim \cite{AK00,AK00b} initiated an extension of the theory to complete metric spaces, including infinite-dimensional Riemannian manifolds. Further developments led to  Wenger's compactness result \cite{Wenger11} which  
will be essential to define spherical Plateau solutions (Subsection \ref{subsection:setup}). In this section, we review the definitions and results developed in \cite{AK00,AK00b,Wenger07,Wenger11,SW11}.

\subsection{Basic definitions}

Let $(E,d)$ be a complete metric space\footnote{The metric $d$ is allowed to take $\infty$ as value.}. Let $n\geq0$, and let $\mathcal{D}^n(E)$ be the set of $(n+1)$-tuples $(f,\pi_1,...,\pi_n)$ of Lipschitz functions on $E$ with $f$ bounded. As a suggestive reference to the finite dimensional theory of currents, $(f,\pi_1,...,\pi_n)$ is also usually denoted by $fd\pi_1\wedge...,\wedge d\pi_n$. Metric currents in the sense of Ambrosio-Kirchheim are a flexible generalization of oriented submanifolds:

\begin{definition} \label{definition current} \cite{AK00}
An $n$-dimensional \emph{metric current} in $(E,d)$ is a multi-linear functional on $\mathcal{D}^n(E)$ such that
\begin{enumerate}
\item If $\pi_i^j$ converges pointwise to $\pi_i$ as $j\to \infty$, and if $\sup_{i,j} \Lip(\pi_i^j) <\infty$, then
$$\lim_{j\to \infty} T(fd\pi_1^j\wedge...,\wedge d\pi_n^j )= T(fd\pi_1\wedge...,\wedge d\pi_n).$$
\item If $\{x\in E; f(x)\neq 0\}$ is contained in the union $\bigcup_{i=1}^n B_i$ of Borel sets $B_i$ and if $\pi_i$ is constant on $B_i$ then
$$T(fd\pi_1\wedge...,\wedge d\pi_n) = 0.$$
\item There exists a finite Borel measure $\mu$ on $E$ such that
$$|T(fd\pi_1\wedge...,\wedge d\pi_n)| \leq \prod_{i=1}^n\Lip(\pi_i) \int_E|f| d\mu$$
for all $fd\pi_1\wedge...,\wedge d\pi_n\in \mathcal{D}^n(E)$.
\end{enumerate}
\end{definition}
The minimal Borel measure $\mu$ satisfying the above inequality is called the \emph{mass} of $T$ and denoted by $\|T\|$. The total mass of $T$ is defined as
$$\mathbf{M}(T) := \|T\|(E),$$
and should be thought of as the ``$n$-dimensional area'' of $T$. 
Currents in the sense of Ambrosio-Kirchheim \cite{AK00} have finite mass by definition (and we will only consider such currents here), but there is a variant of this theory due to Lang \cite{Lang11} which avoids the finite mass condition. 
The \emph{support} $\spt(T)$ of $T$ is the support of the measure $\|T\|$ in the usual sense. The \emph{canonical set} of $T$, called  $\set(T)$, is the collection of points in $E$ with a positive lower density:  
$$\set(T):=\{p\in E; \lim_{r\to 0^+} \|T\|(B(p,r)) r^{-n}>0\}.$$ 
In general, 
$$\set(T) \subset \spt(T)$$ and the inclusion is usually strict.
The \emph{boundary} $\partial T$ is defined by 
$$\partial T(fd\pi_1\wedge...,\wedge d\pi_{n-1}) :=T(1 df\wedge d\pi_1\wedge...,\wedge d\pi_n)$$ for all $fd\pi_1\wedge...,\wedge d\pi_{n-1}\in \mathcal{D}^{n-1}(E)$. The \emph{push-forward} of $T$ by a Lipschitz map $\psi$  from $E$ to another complete metric space $E'$ is given by 
$$\psi_\sharp T(fd\pi_1\wedge...,\wedge d\pi_n) := T(f\circ \psi d(\pi_1 \circ \psi) \wedge...\wedge d(\pi_n\circ \psi))$$ for all $fd\pi_1\wedge...,\wedge d\pi_{n}\in \mathcal{D}^n(E')$. There is also a notion of \emph{restriction} of $T$ to a Borel subset $A\subset E$:
$$(T\llcorner A)(fd\pi_1\wedge...,\wedge d\pi_n) := T(f\chi_Ad\pi_1\wedge...,\wedge d\pi_n)$$
where $\chi_A$ is the characteristic function of $A$ (the above is well-defined by an extension of the functional $T$).

\subsection{Rectifiable sets and integral currents}

We are mainly interested in integral currents, which roughly speaking are currents $T$ such that both $T$ and $\partial T$ are the push-forward of a countable union of elementary currents by Lipschitz maps. An elementary current is obtained as follows: consider $\theta\in L^1(A;\mathbb{N})$ where $A\subset \mathbb{R}^n$, then define the following current in $\mathbb{R}^n$: for all $fd\pi_1\wedge...,\wedge d\pi_n\in \mathcal{D}^n(\mathbb{R}^n)$,
$$\llbracket \theta \rrbracket (fd\pi_1\wedge...,\wedge d\pi_n) : = \int_{A} \theta fd\pi_1\wedge...,\wedge d\pi_n.$$
Integer rectifiable currents and integral currents enjoy the following characterizations \cite[Theorem 4.5, Theorem 8.6]{AK00}, which we will take as definitions: 
\begin{definition} \cite{AK00}
A current $T$ in $E$ is an $n$-dimensional \emph{integer rectifiable current} if and only if there are Lipschitz maps $\varphi_i : A_i \to E $ where $A_i\subset \mathbb{R}^n$ are precompact Borel measurable and have disjoint images by $\varphi_i$, and there are $\theta_i \in L^1(A_i;\mathbb{N})$ such that 
$$T = \sum_{i=1}^\infty (\varphi_i)_\sharp \llbracket \theta_i \rrbracket, \quad \text{and}\quad \|T\| = \sum_{i=1}^\infty \|(\varphi_i)_\sharp \llbracket \theta_i \rrbracket\|.$$
The pair $(\{\varphi_i : A_i\to E\},\{\theta_i\})$ is called a parametrization.

The current $T$ is an \emph{integral current} if and only if both $T$ and $\partial T$ are integer rectifiable currents.
\end{definition}

For instance, if $(M,g)$ is a complete oriented Riemannian manifold with compact boundary and finite volume, then it carries a natural $n$-dimensional integral current usually denoted by $\llbracket 1_M \rrbracket$, induced by ``integration on $M$''.

Recall that a Borel set $S\subset E$ is countable $\mathcal{H}^n$-rectifiable if there is a sequence of Lipschitz functions $\varphi_i:A_i\to E$ where $A_i\subset \mathbb{R}^n$ is Borel, such that 
\begin{equation} \label{rectifiable}
\mathcal{H}^n(S\setminus \bigcup_i \varphi_i(A_i)) = 0
\end{equation}
where $\mathcal{H}^n$ denotes the $n$-dimensional Hausdorff measure. 
It is proved in \cite[Theorem 4.6]{AK00} that if $T$ is an $n$-dimensional integral current with a parametrization $(\{\varphi_i\},\{\theta_i\})$, then
$$\mathcal{H}^n\big(\set(T) \setminus \bigcup_{i=1}^\infty \varphi_i(A_i) \cup \bigcup_{i=1}^\infty \varphi_i(A_i)\setminus \set(T)\big) = 0.$$
The canonical set $\set(T)$ is in particular a countably $\mathcal{H}^n$-rectifiable set in $E$ (contrarily to the support $\spt(T)$, in general). The functions $\theta_i$ determine a Borel function called the multiplicity function $\theta_T : E\to \mathbb{N}$, which is  well-defined $\mathcal{H}^n$-almost everywhere.
In the special case where $E$ is an infinite Riemannian dimensional manifold (i.e. locally modelled on a Hilbert space), it is shown in \cite[Theorem 9.5]{AK00} that
$$\|T\| = \theta_T \mathcal{H}^n \llcorner \set(T).$$
That intuitive formula needs a correction factor in the case of general Banach spaces.

\subsection{Weak and flat topology}
There are two fundamental notions of convergence for integral currents in a metric space: the weak and flat convergences. A sequence $\{T_m\}$ of $n$-dimensional integral currents in $E$ is said to converge \emph{weakly} to some current $T$ if for all $fd\pi_1\wedge...,\wedge d\pi_n\in \mathcal{D}^n(E)$,
$$\lim_{m\to \infty} T_m(fd\pi_1\wedge...,\wedge d\pi_n) =T(fd\pi_1\wedge...,\wedge d\pi_n).$$
In that case, an important result states that such a limit $T$ is also an integral current \cite[Theorem 8.5]{AK00}. Besides, a fundamental property of the mass which makes it so useful in Plateau problems is that it is lower semicontinuous with respect to weak converge, see paragraph below \cite[Definition 3.6]{AK00}: if $T_m$ converges weakly to $T$ then 
$$\mathbf{M}(T) \leq \liminf_{m\to \infty} \mathbf{M}(T_m).$$
The sequence $\{T_m\}$ is said to converge to $T$ \emph{in the flat topology} if there  are sequences $\{U_m\}, \{V_m\}$ of integral currents such that
$$T_m-T =  U_m +\partial V_m, \quad \lim_{m\to \infty } \mathbf{M}(U_m) =  \lim_{m\to \infty } \mathbf{M}(V_m) = 0.$$
Convergence in the flat topology implies convergence in the weak topology.  A partial converse 
is proved in \cite{Wenger07}.

\subsection{The area and coarea formulas} \label{subsection:area}

Next we recall some versions of two particularly useful tools, the area and coarea formulas for countably $\mathcal{H}^n$-rectifiable sets. In this paragraph we assume for simplicity that $(E,d)$ is a separable complete infinite-dimensional Riemannian manifold, embedded isometrically inside an $\ell^\infty$ space $Y$ by a Kuratowski embedding, since only that case is needed here. Note that an $\ell^\infty$ space is a $w^*$-separable dual space in the sense of \cite{AK00}. Given a Lipschitz map $\psi :A\to E$ where $A\subset \mathbb{R}^n$ is Borel, for almost every $y\in A$, $\psi$ is $w^*$-differentiable at $y$ with a $w^*$-differential called $wd_y\psi$ in the sense of \cite{AK00b}\cite[Section 9]{AK00}. The latter is a linear map from $T_y\mathbb{R}^n$ to $Y$.  For almost every $y\in A$, $wd_y\psi$ is of full rank; in that case the image $Q:=wd_y\psi(T_y\mathbb{R}^n)$ is called an approximate tangent $n$-plane at $p:=\psi(y)$. Such $Q$ is a linear $n$-plane inside $Y$, and in our case it is also a tangent $n$-plane of the manifold $E$. More generally, let $S\subset E$ be a countably $\mathcal{H}^n$-rectifiable set, with $\{\varphi_i:A_i\to E\}$ as in (\ref{rectifiable}). At $\mathcal{H}^n$-almost every $p\in S$, there are $i$ and $y\in A_i$ such that $\varphi_i(y)=p$ and an approximate tangent $n$-plane $Q$ at $p$ exists in the sense above.

Let $S$ be a countably $\mathcal{H}^n$-rectifiable set in $E$.
Given a Lispchitz map $g:S \to E'$ where $E'$ is another separable complete infinite-dimensional Riemannian  manifold embedded isometrically in an $\ell^\infty$ space $Y'$,  there is a well-defined nonnegative number for $\mathcal{H}^n$-almost every $z \in S$, called the Jacobian of $g$ and denoted by $\mathbf{J}_n(d^Sg_z)$, such that the following \emph{area formula} \cite[Theorem 8.2]{AK00b} holds for any  Borel function $\theta:S\to [0,\infty]$:
\begin{equation} \label{area form}
\int_{S}\theta(p) \mathbf{J}_n(d^Sg_p) d\mathcal{H}^n(p) = \int_{E'} \sum_{p\in S\cap g^{-1}(z)} \theta(p) d\mathcal{H}^n(z).
\end{equation}
More concretely, suppose that for $\{\varphi_i:A_i\to E\}$ as (\ref{rectifiable}), for $i$ and $y\in A_i$, $p:=\varphi_i(y)\in S$, and suppose that $\varphi_i$ is $w^*$-differentiable at $y$ with a $w^*$-differential of full rank, and the composition $g\circ \varphi_i : \mathbb{R}^n\to E'$ is $w^*$-differentiable at $y$ (all of this holds for almost every $y\in A_i$). Let 
$$Q := wd_y\varphi_i(T_y\mathbb{R}^n)$$ be the tangent $n$-plane at $p$. Then $g$ is tangentially differentiable at $p$ along $Q$ with tangential differential $d^{S}g_p : Q \to T_{g(p)} E'$ satisfying:  
$$d^{S}g_p = wd_y(g\circ \varphi_i) \circ (wd_y\varphi_i)^{-1}.$$
In our case where $E,E'$ are infinite-dimensional Riemannian manifolds, $Q$ and $T_{g(p)} E'$ are Hilbert linear spaces.
The Jacobian $\mathbf{J}_n(dg_p)$ is then equal to the absolute value of the usual Jacobian determinant of the linear map $d^{S}g_p$.
In particular when $g$ is $\lambda$-Lipschitz then 
$$\mathbf{J}_n(dg_z)\leq \lambda^n$$ as expected. Given a tangent $n$-plane $Q$ of $S$ at $p$ as above, we often use the following more classical notations:
$$dg\big|_Q := d^{S}g_p,$$
$$|\Jac g\big|_Q| := \mathbf{J}_n(d^Sg_p).$$

Let $S$ be a countably $\mathcal{H}^n$-rectifiable set in $E$.
Consider a Lipschitz function $\pi:S\to \mathbb{R}^k$ where $k\leq n$. At $\mathcal{H}^n$-almost every $p\in S$, there is a tangent $n$-plane $Q$
along which the tangential differential $d^{S}\pi_p$ exists. At such $p$, there is a nonnegative number called coarea factor and denoted by $\mathbf{C}_k(d^{S}\pi_p)$ such that the following \emph{coarea formula} \cite[Theorem 9.4]{AK00b} holds for any  Borel function $\theta:S\to [0,\infty]$:
\begin{equation} \label{coarea form}
\int_S\theta(p) \mathbf{C}_k(d^S\pi_p)d\mathcal{H}^n(p) = \int_{\mathbb{R}^k}\big( \int_{\pi^{-1}(z)} \theta(p) d\mathcal{H}^{n-k}(p)  \big)d\mathcal{H}^k(z).
\end{equation}
Besides, for $\mathcal{H}^k$-almost every $z\in \mathbb{R}^k$, the set $\pi^{-1}(z)\cap S$ is countably $\mathcal{H}^{n-k}$-rectifiable.
The precise definition of the coarea factor $\mathbf{C}_k(d\pi_z)$ is a bit more involved than the Jacobian \cite[Definition 9.1]{AK00b}, nevertheless it is similar to what appears in the smooth finite dimensional case  and when $\pi$ is $\lambda$-Lipschitz then 
$$\mathbf{C}_k(d\pi_z)\leq \lambda^k$$ as expected.

\subsection{The slicing theorem for integral currents} \label{subsection:slicing}

We move on to the slicing theorem for integral currents, which enables to construct lower dimensional integral currents out of a given integral current and a Lipschitz map. Again, we still assume for simplicity that $E$ is an infinite-dimensional Riemannian manifold (locally modelled on a Hilbert space).
Let $S$ be a countably $\mathcal{H}^n$-rectifiable set in $E$.
Given $fd\pi_1\wedge ...\wedge d\pi_n\in \mathcal{D}^n(E)$, then at $\mathcal{H}^n$-almost every $p\in S$, an approximate tangent $n$-plane $Q$ exists, the tangential differentials $d^S (\pi_j)_p$ along $Q$ exist and the $n$-covector $fd^S\pi_1\wedge ...\wedge d^S\pi_n$ is well-defined.
There is a notion of orientations on $S$ \cite[Section 9]{AK00}. Roughly speaking, an orientation $\tau$ on $S$ endows each approximate tangent $n$-plane $Q$ of $S$ with an $n$-vector $\mathbf{e}_1\wedge...\wedge \mathbf{e}_n$ where $(\mathbf{e}_1,...,\mathbf{e}_n)$ is an orthonormal basis of $Q$ (recall that here $E$ is a manifold).

Let $T$ be an $n$-dimensional integral current in $E$, with multiplicity function $\theta_T$. 
There is an intrinsic description of $T$, based on the notion of orientation \cite[Theorem 9.1]{AK00}. In fact there is an orientation $\tau_T$ on $\set(T)$ such that for all $fd\pi_1\wedge ...\wedge d\pi_n\in \mathcal{D}^n(E)$, 
$$T(fd\pi_1\wedge ...\wedge d\pi_n) = \int_{\set(T)} f(p) \theta_T(p) \langle d^{\set(T)}\pi_1\wedge ...\wedge d^{\set(T)}\pi_n, \tau_T \rangle  d\mathcal{H}^n(p).$$
Conversely if $S$ is a countably $\mathcal{H}^n$-rectifiable set in $E$, if we have $\theta:E\to \mathbb{N}$ such that  $\int_S\theta d\mathcal{H}^n<\infty$, and an orientation $\tau$ on $S$ then the current 
$$\llbracket S, \theta, \tau \rrbracket$$ defined by the analogue of the formula above is an $n$-dimensional integral current.

Now consider a Lipschitz map $\pi : \spt(T) \to \mathbb{R}^k$ where $k\leq n$. Then for each $x\in \mathbb{R}^k$, there is an integral current $\langle T,\pi,x\rangle$ called \emph{sliced current} \cite[Theorems 5.6 and 5.7]{AK00}, characterized by the fact that for all $fd\pi_1\wedge ...\wedge d\pi_{n-k} \in \mathcal{D}^{n-k}(E)$ and $\psi \in C_c(\mathbb{R}^k)$,
\begin{equation} \label{slicing thm}
\big[\int_{\mathbb{R}^k} \langle T,\pi,x\rangle \psi(x) dx \big](fd\pi_1\wedge ...\wedge d\pi_{n-k}) = T(f.(\psi \circ \pi) d\pi \wedge d\pi_1...\wedge d\pi_{n-k})
\end{equation}
where if $\pi =(u_1,...,u_k)$ and $u_j$ are the coordinate functions, then $d\pi := du_1\wedge...\wedge du_k$. Besides, with the notation $T = \llbracket \set(T), \theta_T, \tau_T\rrbracket $ defined above,  for almost every $x\in \mathbb{R}^k$ we have an orientation $\tau_x$ of $\set(T) \cap \pi^{-1}(x)$ such that
\begin{equation} \label{slicing thm bis}
\langle T,\pi,x\rangle = \llbracket \set(T)\cap \pi^{-1}(x), \theta_T, \tau_x \rrbracket,
\end{equation} 
see \cite[Theorem 9.7]{AK00}.
The fact that
$$\spt(\partial \langle T,\pi,x\rangle ) \subset \spt(\partial T),$$ follows from \cite[Theorem 5.6 (i)]{AK00} applied to the boundary $\partial T$, or alternatively by iterating \cite[Theorem 5.6 (iii), Lemma 5.3]{AK00}.

\subsection{Integral current spaces and intrinsic flat topology}

We end this section with the definitions of integral current spaces, the intrinsic flat topology of Sormani-Wenger, and Wenger's compactness theorem \cite{SW11}.

\begin{definition}
An \emph{integral current space} of dimension $n$ is a triple 
$$C:=(X,d,T)$$ where $(X,d)$ is a metric space with completion called $\overline{X}$, and $T$ is an $n$-dimensional integral current  in $\overline{X}$, such that $\set(T)=X$. The mass $\mathbf{M}(C)$ is by definition $\mathbf{M}(T)$. 
\end{definition}

A simple example of integral current space is given by a complete, oriented Riemannian $n$-manifold $(M,g)$ with compact boundary and finite volume: the metric space is $M$ endowed with the geodesic distance $\dist_g$ induced by $g$, and the integral current structure $\llbracket 1_M \rrbracket$  is the natural integral current induced by integration on $M$. Recall that we allow the metric to assume $\infty$ as value, and so we allow $M$ to be non-connected.
In general, an integral current in a metric space $E$ determines uniquely an integral current space. The boundary $\partial C$ is the integral current space induced by $\partial T$. 

If $C:=(X,d,T)$, $C':=(X',d',T')$ are two integral current spaces such that there is an isometry $\varphi:\overline{X} \to \overline{X'}$ with $\varphi_\sharp T = T'$ (such a $\varphi$ is called a current preserving isometry \cite[Definition 3.26]{SW11}), then we say that $C$ and $C'$ are \emph{isomorphic}.

In the same way that Gromov-Hausdorff convergence generalizes Hausdorff convergence by allowing arbitrary isometric embeddings in a complete metric space, the notion of \emph{intrinsic flat convergence} generalizes flat convergence as follows \cite[Theorem 4.2]{SW11}.
\begin{definition}
 A sequence of integral current spaces $\{(X_m,d_m,S_m)\}_{m\geq 0}$ converges in the \emph{intrinsic flat topology} to an integral current space $(X_\infty,d_\infty,S_\infty)$ when there is a complete metric  space $\mathbf{Z}$ and there are isometric embeddings 
$$j_m : \overline{X}_m \to \mathbf{Z}, \quad j_\infty:\overline{X}_\infty \to \mathbf{Z}$$ such that $(j_m)_\sharp(S_m)$ converges in the flat topology to $(j_\infty)_\sharp(S_\infty)$ inside $\mathbf{Z}$.
\end{definition}

The complete metric space $\mathbf{Z}$ can be taken to be a Banach space if the metrics $d_m$, $d_\infty$ only assume finite values, by the standard Kuratowski embedding. Actually the above notion of convergence is induced by a metric called the \emph{intrinsic flat distance} $\mathbf{d}_\mathcal{F}$ and defined in \cite[Definition 1.1]{SW11}: given two integral current spaces of dimension $n$, $C:=(X,d,S)$ and $C':=(X',d',S')$, set
$$\mathbf{d}_\mathcal{F}(C,C') := \inf\{\mathbf{M}(U)+\mathbf{M}(V)\}$$
where the infimum is taken over all complete metric spaces $(Z,d)$ and all integral currents $U$, $V$ in $Z$ such that there are isometric embeddings
$\varphi :(\overline{X},d) \to Z$, $\varphi' :(\overline{X'},d') \to Z$ with
$$\varphi_\sharp(S) - (\varphi')_\sharp(S') = U+\partial V.$$
By \cite[Theorem 3.27]{SW11}, $\mathbf{d}_\mathcal{F}(C,C') =0$ if and only if $C$ and $C'$ are isomorphic.
 
One of the fundamental properties of integral current spaces is the following compactness theorem \cite{Wenger11}\cite[Theorem 4.19]{SW11}: 
\begin{theo} [\cite{Wenger11,SW11}]
If for some constant $c$, $\{(X_m,d_m,S_m)\}_{m\geq 0}$ is a sequence of $n$-dimensional integral current spaces with 
$$\mathbf{M}(S_m) +\mathbf{M}(\partial S_m) \leq c <\infty,$$
$$\diam(\spt(S_m)) \leq c,$$
then there is a subsequence $\{(X_{m_k},d_{m_k},S_{m_k})\}_{k\geq 0}$ converging in the intrinsic flat topology to an $n$-dimensional integral current space $(X_\infty,d_\infty,S_\infty)$. 
\end{theo}
The above compactness result will be necessary when defining spherical Plateau solutions in Subsection \ref{subsection:setup}.

\subsection{Approximation of integral currents by polyhedral chains} \label{poly chain approx}

We end this section with a useful approximation theorem for integral currents in spherical manifolds by polyhedral chains. This result can be simply deduced from the analogous result in finite dimensions, which in turn is completely standard.

Let $(N,g_N)$ be an infinite-dimensional Riemannian manifold which is locally isometric to the unit sphere of an infinite-dimensional Hilbert space. 
In our context, a $k$-dimensional integral current $P$ in $(N,g_N)$ is called a \emph{polyhedral chain} of dimension $k$ if there are smoothly embedded totally geodesic $k$-simplices $S_1,...,S_m\subset  N$ endowed with an orientation, and integers $a_j$ so that
$$P= \sum_{j=1}^m  a_j \llbracket 1_{S_j}\rrbracket.$$
\begin{lemme} \label{approx polyh}
For any $\epsilon>0$, and any integral current $C$ with compact support in $(N,g_N)$, there is a polyhedral chain $P$ such that 
\begin{itemize}
\item $C$ and $P$ are $\epsilon$-close in the flat topology, 
\item $\spt(C)$ and $\spt(P)$ are $\epsilon$-close in the Hausdorff topology,
\item $|\mathbf{M}(C)- \mathbf{M}(P) |\leq \epsilon$. 
\end{itemize}
\end{lemme}

Here is an outline of proof. Consider an open ball $B$ centered at the origin in the Hilbert space $\ell^2(\mathbb{N})$ where $\mathbb{N}$ is the set of natural numbers, and consider an integral current $C$ with compact support in $B$. For any integer $L\geq 1$, we can use the orthogonal projection onto $\ell^2(\{1,...,L\})$ to map $C$ to an integral current $C_1$ compactly supported  inside the finite dimensional space $B\cap \ell^2(\{1,...,L\})$. If $L$ is chosen large enough, $C_1$ will be as close to $C$ as we wish in the sense of Lemma \ref{approx polyh}. Now $C_1$ is also an integral current in the sense of Federer-Fleming (see the appendix of \cite{AK00}), so the usual approximation  results (\cite[Sections 4.1 and 4.2]{Federer69}, \cite{DePauw14})
in finite dimension can be applied and give the desired polyhedral chain approximation $P_1$ in  $B\cap \ell^2(\{1,...,L\})$. This construction clearly generalizes to small balls in $(N,g_N)$ if $N$ is separable, which can always be assumed to be true since $\spt(C)$ is compact. In the general case where the support of $C$ is not contained in a small ball of $N$, one can argue using a partition of unity and construct via an interpolation map a current $C_2$ arbitrarily close to $C$ in the sense of Lemma \ref{approx polyh}, which is locally finite dimensional in the following sense: for any $x\in \spt(C_2)$, there is a ball $B_x$ containing $x$ such that $\spt(C_2)\cap B_x$ is contained in a totally geodesic finite dimensional plane of $B_x$. The rest is standard as before.

\section{Preliminaries on the barycenter map} \label{appendix b}

In \cite{BCG95,BCG96}, Besson-Courtois-Gallot proved that the normalized volume entropy on a closed hyperbolic manifold of dimension at least $3$ is uniquely achieved at the hyperbolic metric. The proof of this striking result relies on the barycenter map. 
In this section\footnote{I would like to thank Cosmin Manea for corrections and useful discussions about this section.}, we define a variant of the barycenter map used in \cite{BCG95,BCG96} (see also \cite{Sambusetti99}), which we will need in the proofs of the uniqueness of spherical Plateau solutions.
Instead of working with $L^2$ functions on a boundary at infinity as in \cite{BCG95},
we directly work with $\ell^2$ functions on the underlying group $\Gamma$. This setup is better adapted to extensions to more general situations ($3$-manifolds, Plateau Dehn fillings). This section only treats hyperbolic manifolds, even though everything carries out more generally  in the locally symmetric rank one case.
All the results here are essentially contained  in \cite{BCG95}.

\subsection{Definition of the barycenter map} 
 
Let $(M,g_0)$ be a closed oriented hyperbolic manifold. Let $(\tilde{M},g_0)$ be its universal cover, namely the hyperbolic $n$-space. 
Let $\Gamma:=\pi_1(M)$. The latter acts properly cocompactly, freely and properly on $(\tilde{M}, g_0)$.
Let $ {S^\infty}$ be the unit sphere in the Hilbert space $\ell^2(\Gamma)$, on which $\Gamma$ acts freely and properly by isometries via the regular representation $\lambda_\Gamma: \Gamma\to \End(\ell^2(\Gamma))$, and let $ {S^\infty}/\lambda_\Gamma(\Gamma)$ be the quotient manifold endowed with the standard round metric (see Subsection \ref{subsection:setup} for more details).

It is well-known that the distance functions on $(\tilde{M},g_0)$ satisfy the following Hessian lower bound: for any fixed $w\in \tilde{M}$,  at any point different from $w$ we have
$$Dd \dist_{g_0}(w,.)
 \geq \Id - d\dist_{g_0}(w,.)\otimes d\dist_{g_0}(w,.).$$
Distance functions are not smooth so in the setting of hyperbolic manifolds, it is more convenient from a technical point of view to work with modified distance functions called $\rho_w$. Fix a smooth strictly convex  increasing function 
$$\varkappa:[0,\infty)\to [0,\infty)$$ such that $\lim_{t\to \infty} t^{-1}\varkappa(t) =1$, $\varkappa'(t)< 1$ for all $t$, and for any $w\in \tilde{M}$, the composition
$$\rho_{w}(.):= \varkappa(\dist_{g_0}(w,.))$$
is smooth everywhere and satisfies 
\begin{equation} \label{hessian lower bd''}
Dd\rho_{w} \geq \Id - d\rho_{w}\otimes d\rho_{w}.
\end{equation}
Such a function exists, for example if we set 
$$\varkappa(t)= \frac{1}{c}\log(\cosh(ct))$$ 
where $c$ is a positive constant, it is an exercise to check that (\ref{hessian lower bd''}) holds whenever $c$ is large enough.

\begin{definition} \label{s+}
Fix a basepoint $o\in \tilde{M}$. Let $\mathbb{S}^+$ be the set of functions in $ {S^\infty}$ with finite support.
For $f\in \mathbb{S}^+$, consider the functional
\begin{equation} \label{bf}
\begin{split}
\mathcal{B}_f & : \tilde{M}\to [0,\infty]\\
\mathcal{B}_f(x) & := \sum_{\gamma\in \Gamma} |f(\gamma)|^2 \rho_{\gamma.o}(x).
\end{split}
\end{equation}
The \emph{barycenter map} is then defined as
$$\mathrm{Bar} : \mathbb{S}^+\to \tilde{M}$$
$$\mathrm{Bar}(f) := \text{ the unique point minimizing $\mathcal{B}_f$}.$$
\end{definition} 
The barycenter map is well-defined: the modified distance functions $\rho_{\gamma.o}$ are strictly convex,  
moreover $\mathcal{B}_f$ tends to infinity uniformly as $x\to \infty$, so that the point where $\mathcal{B}_f$ attains its minimum exists and is unique. 
The subset $\mathbb{S}^+\subset  {S^\infty}$ is invariant by $\Gamma$, and $\mathrm{Bar}$ is $\Gamma$-equivariant. 
The quotient map 
$$\mathbb{S}^+/\lambda_\Gamma(\Gamma) \to M$$ is also denoted by $\mathrm{Bar}$.

\subsection{Notations and a priori Lipschitz bounds}

The regularity of the barycenter map 
$$\mathrm{Bar}:\mathbb{S}^+/\lambda_\Gamma(\Gamma) \to M$$
is not completely immediate. To avoid discussing such issues, we will only consider the barycenter map restricted to the support of polyhedral chains in $\mathbb{S}^+/\lambda_\Gamma(\Gamma)$. Recall that by definition (see Subsection \ref{poly chain approx}), a $k$-dimensional polyhedral chain $P$ in $\mathbb{S}^+/\lambda_\Gamma(\Gamma)$ is a $k$-dimensional integral current that can be written as
$$P= \sum_{j=1}^m  a_j \llbracket 1_{S_j}\rrbracket$$
where $a_j$ are integers and 
$$S_1,...,S_m\subset\mathbb{S}^+/\lambda_\Gamma(\Gamma)\subset  {S^\infty}/\lambda_\Gamma(\Gamma)$$ are finitely many smoothly embedded totally geodesic $k$-simplices endowed with an orientation. In particular the support of a polyhedral chain is by definition a finite union of totally geodesic simplices. Given a polyhedral chain $P$ as above, each simplex $S_j$ lifts to a simplex $\tilde{S}_j$ in $\mathbb{S}^+ \subset \ell^2(\Gamma)$. 
We claim that there is a finite set $\mathbf{S}_j \subset \Gamma$ depending only on $\tilde{S}_j$ such that any element of  $\tilde{S}_j$ is a function with support in $\mathbf{S}_j$.
Indeed, the $k$-simplex $\tilde{S}_j$ is the convex hull of its extremal points $f_0,...,f_k \in \mathbb{S}^+$. If $\mathbf{S}_j$ denotes the union of the finite supports of $f_0,...,f_k$, then any linear combination of those functions has support contained in $\mathbf{S}_j$, and that proves the claim.
Now given a polyhedral chain $P$ in $\mathbb{S}^+/\lambda_\Gamma(\Gamma)$, one can check without difficulty that the restriction
$$\mathrm{Bar}: \spt(P) \to M$$
is continuous (and smooth on each simplex by the discussion below).

Consider $f\in \mathbb{S}^+$. For $v\in T_x\tilde{M}$, set 
  \begin{equation}\label{def de h}
  H_f(v,v) := \sum_\gamma  f^2(\gamma)  |d_x\rho_{\gamma.o}(v)|^2.
  \end{equation}
 The endomorphism $H_f$ is symmetric, it satisfies
 $$\tr H_f < 1$$
due to the fact that $|\nabla \rho_{\gamma.o}|<1$, and has eigenvalues 
$$0\leq \mu_1(f)\leq ...\leq \mu_n(f) < 1.$$
In particular $\Id-H_f$ has strictly positive determinant.
For $v\in T_x\tilde{M}$, set 
 $$K_f(v,v) := \sum_\gamma  f^2(\gamma) Dd_x\rho_{\gamma.o}(v,v) .$$
By strict convexity of the modified distance functions,
\begin{equation} \label{k>0}
 \forall  v\in T_x\tilde{M} \setminus \{0\},\quad K_f(v,v) >0.
\end{equation}
The barycenter $x$ of $f$ is characterized by the equation
\begin{equation}\label{implicit f}
\sum_{\gamma\in \Gamma} f^2(\gamma)  d_x\rho_{\gamma.o}(.)=0\in T_x^*\tilde{M}.
\end{equation}

For $1\leq k\leq n$, let $\xi$ be a totally geodesic $k$-simplex contained in  $\mathbb{S}^+$, embedded in $ {S^\infty}$, passing through $f\in \mathbb{S}^+$. Then the restriction of 
$\mathrm{Bar}$ to $\xi$ is smoothly differentiable around $f$ 
because of the implicit function theorem  and (\ref{k>0}), (\ref{implicit f}). 
Let $Q$ be the tangent $n$-plane of $\xi$ at $f$.  
The differential of $\mathrm{Bar}$ along $Q$  is denoted by 
$d\mathrm{Bar}\big|_Q: Q\to T_{\mathrm{Bar}(p)} \tilde{M}$. 
Sometimes we will drop the subscript $|_Q$ when the choice of tangent $k$-plane is clear.
By differentiating (\ref{implicit f}) with respect to $f$ we obtain for all  $\dot{f}\in Q$:
$$\sum_{\gamma\in \Gamma} 2f(\gamma)\dot{f}(\gamma) d_x\rho_{\gamma.o}(.) + \sum_{\gamma\in \Gamma} f^2(\gamma) Dd_x\rho_{\gamma.o}(d \mathrm{Bar}\big|_Q(\dot{f}),.)=0.$$
After applying Cauchy-Schwarz,
we get the following: for all $v\in T_x\tilde{M}$ and $\dot{f}\in Q$, with $\|\dot{f}\|_{\ell^2}=1$,
  \begin{equation} \label{ja}
 K_f(d \mathrm{Bar}\big|_Q(\dot{f}), v) \leq 2[H_f(v,v)]^{1/2}
 \end{equation}
 and
 \begin{equation} \label{jacoco}
  |\Jac \mathrm{Bar}\big|_Q| \leq 2^n\frac{(\det H_f)^{1/2}}{\det K_f} \leq 2^n\frac{(\det H_f)^{1/2}}{\det(\Id - H_f)}.
  \end{equation}
 Going from (\ref{ja}) to the first inequality in (\ref{jacoco}) is an application of the Gram-Schmidt orthonormalization process for matrices, see the proof of \cite[Lemma 5.4]{BCG96}. The second inequality in (\ref{jacoco}) is a consequence of the inequality
 \begin{equation} \label{kfgeq}
 K_f  \geq  \Id-H_f 
\end{equation}
which in turn follows from (\ref{hessian lower bd''}).

For any $f\in \mathbb{S}^+$, let $\mu_1(f)\leq ...\leq \mu_n(f)$ be the eigenvalues of the endomorphism $H_f$ defined in (\ref{def de h}). 
The following is an a priori Lipschitz bound corresponding to \cite[Lemma 7.5.a]{BCG95}.
\begin{lemme} \label{lipschitz barycenter map}

Let $\kappa>0$ and let $\alpha \subset \mathbb{S}^+$ be a connected continuous piecewise geodesic curve.
Suppose that for all $f\in \alpha$, $\mu_n(f)\leq 1-\kappa/2$. Then 
$$\length_{g_0}(\mathrm{Bar}(\alpha)) \leq K_1\length(\alpha)$$
for a constant $K_1$ depending only on $\kappa$.

\end{lemme}
\begin{proof}

Given $f\in \alpha$, let $V$ be the tangent $1$-plane of $\alpha$ at $f$, let $\dot f \in V$ with $\|\dot f\|_{\ell^2}=1$, and suppose that $d_f\mathrm{Bar}(\dot f) \neq 0$. Set $v:= \frac{d_f\mathrm{Bar}(\dot f)}{|d_f\mathrm{Bar}(\dot f)|}\in T_x\tilde{M}$.

By (\ref{ja}) and (\ref{kfgeq}), 
\begin{equation}\label{13:31}
|d_f\mathrm{Bar}(\dot{f})| \leq 2 \frac{(H_f(v,v))^{1/2}}{1-H_f(v,v)}.
\end{equation}
The lemma readily follows from integrating (\ref{13:31}) along $\alpha$.


\end{proof}

Here is another a priori bound corresponding to \cite[Lemma 7.5.b]{BCG95}.
\begin{lemme} \label{fact 2}
Let $f,f'\in \mathbb{S}^+$ and let $\beta$ be the geodesic segment joining $\mathrm{Bar}(f)$ and $\mathrm{Bar}(f')$. 
Let $P$ be the parallel transport from $\mathrm{Bar}(f)$ to $\mathrm{Bar}(f')$ along  $\beta$.
Then 
$$\|H_{f'}\circ P - H_f\|\leq K_2(\length_{g_0}(\beta) +\|f-f'\|_{\ell^2})$$
for a constant $K_2$.
\end{lemme}
\begin{proof}
Let $Y_0$ be a unit norm tangent vector of $M$ based at $\mathrm{Bar}(f)$ and let $Y_2$ be its parallel transport at $\mathrm{Bar}(f')$ along $\beta$. Then, writing $x:=\mathrm{Bar}(f), x':=\mathrm{Bar}(f')$,
\begin{equation} \label{hf'}
\begin{split}
& |H_{f'}(Y_2,Y_2) - H_f(Y_0,Y_0) |  \\
= & |\sum_\gamma (f')^2(\gamma) |d_{x'}\rho_{\gamma.o}(Y_2)|^2 - \sum_\gamma  f^2(\gamma)  |d_x\rho_{\gamma.o}(Y_0)|^2  |\\
 \leq &  |\sum_\gamma f^2(\gamma)   \big( |d_{x'}\rho_{\gamma.o}(Y_2)|^2 - |d_x\rho_{\gamma.o}(Y_0)|^2\big) | \\ &  + |\sum_\gamma  \big((f')^2(\gamma) -f^2(\gamma)\big)  |d_{x'}\rho_{\gamma.o}(Y_2)|^2 |.
\end{split}
\end{equation}
The Hessian of the smooth modified distance functions $\rho_{\gamma.o}$ is uniformly bounded from above. This bound controls uniformly the terms $\big( |d_{x'}\rho_{w}(Y_2)|^2 - |d_x\rho_{w}(Y_0)|^2\big)$ by integration along $\beta$, and we conclude the proof with Cauchy-Schwarz.

\end{proof}

\subsection{The Jacobian bound}

We are now ready to state the main estimates for the Jacobian of the barycenter map, which is sharp and is the key technical point in \cite{BCG95,BCG96}.

 \begin{lemme} \cite{BCG95} \label{astuce}
 Suppose that $n\geq 3$. Let $f\in \mathbb{S}^+$
 and let $Q$ be the tangent $n$-plane at $f$ of a totally geodesic $n$-simplex in $ \mathbb{S}^+$  passing through $f$. Then 
 \begin{equation} \label{jacbar<}
 |\Jac \mathrm{Bar}\big|_Q| \leq \big(\frac{{4n}}{(n-1)^2}\big)^{n/2}.
 \end{equation}
 
Moreover for any $\eta>0$ small enough,
there exists $c_{\eta}>0$ with $\lim_{\eta \to 0} c_{\eta} = 0$, such that the following holds. 
If 
 $$  |\Jac \mathrm{Bar}\big|_Q| \geq \big(\frac{{4n}}{(n-1)^2}\big)^{n/2} -\eta,$$
then for any norm $1$ tangent vector $\Vec{u}\in Q$,
 \begin{equation} \label{dbar>}
 |d \mathrm{Bar}\big|_Q(\Vec{u})| \geq  \big(\frac{4n}{(n-1)^2}\big)^{1/2} - c_{\eta}
 \end{equation}
 and
 for any
connected continuous piecewise geodesic curve $\alpha \subset \mathbb{S}^+$ of length less than $\eta$ starting at $f$, 
 we have 
  \begin{equation} \label{length<}
 \length_{g_0}(\mathrm{Bar}(\alpha)) \leq   (\big(\frac{4n}{(n-1)^2}\big)^{1/2}+  c_{\eta}) \length(\alpha).
  \end{equation}
 \end{lemme}
 \begin{proof}
The upper bound for the Jacobian is proved by combining (\ref{jacoco}) and \cite[Proposition B.1]{BCG95}. 
 
Let us move on to the second part of the lemma.  
As before, denote by $0\leq \mu_1(f)\leq ...\leq \mu_n(f)< 1$ the eigenvalues of $H_f$.  
Let $\theta>1$ be such that $\tr (\theta H_f)=1$.
By the key bound in \cite[Proposition B.5]{BCG95}, there is a universal constant $A>0$ such that
$$2^n \frac{(\det H_f)^{1/2}}{\det \big(\Id-H_f \big) }\leq \big(\frac{4n}{(n-1)^2}\big)^{n/2} \big(1 - A\sum_{j=1}^n (\theta \mu_j(f) -\frac{1}{n})^2\big).$$
Hence for any $\eta'>0$ if $  |\Jac \mathrm{Bar}\big|_Q| \geq \big(\frac{{4n}}{(n-1)^2}\big)^{n/2} -\eta$ for a small $\eta$, then 
\begin{equation} \label{fait hf0}
\mu_n(f)\leq  \frac{1}{n}+ \eta'.
\end{equation}
and so by (\ref{ja}) we have 
for all norm $1$ tangent vector $\Vec{u}\in Q$:
\begin{equation} \label{c''eta}
|d \mathrm{Bar}\big|_Q(\Vec{u})|  \leq \big(\frac{4n}{(n-1)^2}\big)^{1/2} +  c''_{\eta}
\end{equation}
where $\lim_{\eta\to 0}c''_{\eta}=0$. Therefore if  $  |\Jac \mathrm{Bar}\big|_Q| \geq \big(\frac{{4n}}{(n-1)^2}\big)^{n/2} -\eta$, then (\ref{c''eta}) forces the following to hold: for all norm $1$ tangent vector $\Vec{u}\in Q$,
 $$|d \mathrm{Bar}\big|_Q(\Vec{u})| \geq  \big(\frac{4n}{(n-1)^2}\big)^{n/2} - c_{\eta}$$
where $\lim_{\eta\to 0}c_{\eta}=0$.

We want to propagate the estimate (\ref{c''eta}) on $\|d\mathrm{Bar}\|$ to a whole ``neighborhood'' of $f$ inside $\mathbb{S}^+$.
Suppose that a continuous piecewise geodesic curve $\alpha \subset \mathbb{S}^+$ joins $f$ to $f'$ and has length less than a number $\eta$. 
Let us check, as an intermediate step, that for $\kappa>0$,
\begin{equation} \label{fait hf}
\text{if  $\mu_n(f)\leq 1- \kappa$, then $\mu_n(f')\leq 1- \kappa + c'_{\eta}$} 
\end{equation}
where the constant $c'_{\eta}$ depends only on $\eta$ and satisfies $\lim_{\eta \to 0} c'_{\eta} = 0$.  Indeed, we can proceed as in \cite[Lemma 7.5]{BCG95}. Let $K_1,K_2$ as in Lemmas \ref{lipschitz barycenter map} and \ref{fact 2}. 
Suppose that $\mu_n(f)\leq 1- \kappa$, and that $\eta$ is small so that there is $K_3$ with $K_2(K_1+1)+1 <  K_3 < \frac{\kappa}{2\eta}$. To argue towards a contradiction, assume that (\ref{fait hf}) is not true and that there is a point $f_1\in \mathbb{S}^+$ joined to $f$ by a continuous piecewise geodesic curve $\alpha \subset \mathbb{S}^+$ of length less than $\eta$, such that 
$$\mu_n(f_1) \geq 1- \kappa + K_3 \eta.$$
By truncating the curve, we can also assume that  $f_1$ is the only point on $\alpha$ which satisfies the above condition, so that $\mu_n(f_0) \leq 1-\kappa/2$ for every $f_0\in \sigma$. We have $\mu_n(f_1) \geq 1- \kappa + K_3 \eta$, but by Lemmas \ref{lipschitz barycenter map} and \ref{fact 2}, we would also have 
$$\mu_n(f_1) \leq \mu_n(f)  + K_2(K_1+1)\eta   \leq 1-\kappa + K_2(K_1+1)\eta < 1-\kappa +  K_3\eta.$$
This is a contradiction, thus (\ref{fait hf}) is checked.


We now conclude using (\ref{fait hf}).
Suppose that a continuous piecewise geodesic curve $\alpha \subset \mathbb{S}^+$ joins $f$ to $f'\in \mathbb{S}^+$ and has length less than a small number $\eta$. Let $\Vec{u}$ be a unit tangent vector of $\alpha$ at a point of $\alpha$. 
By (\ref{ja}), (\ref{kfgeq}), (\ref{fait hf0}) and (\ref{fait hf}),
$$
| d \mathrm{Bar}(\Vec{u})| \leq (\frac{4n}{(n-1)^2}\big)^{1/2} +  c_{\eta}$$
where $c_{\eta}$ is a constant depending only on $\eta>0$ such that $\lim_{\eta\to 0} c_{\eta}=0$. We choose $\eta$ small enough and we integrate the previous inequality from $f$ to $f'$
along the points of $\alpha$ where the tangent vector exists and the differential of $\mathrm{Bar}$ is well-defined.
We readily obtain the desired conclusion:
 $$\length_{g_0}(\mathrm{Bar}(\alpha)) \leq   (\big(\frac{4n}{(n-1)^2}\big)^{1/2}+  c_{\eta}) \length(\alpha).$$

 \end{proof}

 \section{The spherical Plateau problem} \label{definition spp}

\subsection{Definitions for the spherical volume and spherical Plateau solutions} \label{definition of spherical plateau problem} \label{subsection:setup}

Let $\Gamma$ be a countable group.
Let $\ell^2(\Gamma)$ be the space of $\ell^2$ real functions on $\Gamma$. Set
$$ {S^\infty}:= \{f:\Gamma\to \mathbb{R};\quad \| f \|_{\ell^2} =1\}.$$
Note that $S^\infty$ may be finite dimensional, and that $S^\infty$ implicitly depends on $\Gamma$.
The $\ell^2$-norm induces a Riemannian metric $\mathbf{g}_{\mathrm{Hil}}$ on the unit sphere $ {S^\infty}$. The group $\Gamma$ acts linearly isometrically on $\ell^2(\Gamma)$, and thus on $ {S^\infty}$, by the (left) regular representation
$$\lambda_\Gamma:\Gamma \to \End(\ell^2(\Gamma))$$
defined as follows: for all $\gamma\in \Gamma, x\in \Gamma, f\in  {S^\infty}$,
$$(\lambda_\Gamma(\gamma).f)(x) := f(\gamma^{-1}x).$$
Denote by 
$$( {S^\infty}/\lambda_\Gamma(\Gamma),\mathbf{g}_{\mathrm{Hil}})$$
the corresponding spherical quotient with the quotient metric. The action of $\Gamma$ on $ {S^\infty}$ is not free exactly when $\Gamma$ has torsion elements. When $\Gamma$ is torsion-free and non-trivial, then it is not hard to check that $  {S^\infty}$ is an infinite-dimensional contractible sphere, $\Gamma$ acts properly freely on $ {S^\infty}$ by the regular representation and the quotient space $ {S^\infty}/\lambda_\Gamma(\Gamma)$ is in fact a classifying space for $\Gamma$ (namely a $K(\Gamma,1)$ space).

Let $n\geq 0$ be an integer. Consider the following homology groups defined using integral currents:
\begin{align*}
\mathcal{Z}_n( {S^\infty}/\lambda_\Gamma(\Gamma))  := \{T; \quad \text{$T$ is an integral $n$-current with} \\
 \text{compact support in $ {S^\infty}/\lambda_\Gamma(\Gamma)$}\}, 
 \end{align*}
\begin{align*}
\mathcal{B}_n( {S^\infty}/\lambda_\Gamma(\Gamma))  := \{\partial D; \quad \text{$D$ is an integral $(n+1)$-current with}\\
  \text{compact support in $ {S^\infty}/\lambda_\Gamma(\Gamma)$}\},
 \end{align*}
 $$
\mathbf{H}_n( {S^\infty}/\lambda_\Gamma(\Gamma))  := \mathcal{Z}_n( {S^\infty}/\lambda_\Gamma(\Gamma))/\mathcal{B}_n( {S^\infty}/\lambda_\Gamma(\Gamma)).
$$
There is a natural morphism 
$$
\hat{\pi}: H_*(\Gamma;\mathbb{Z}) \to \mathbf{H}_*( {S^\infty}/\lambda_\Gamma(\Gamma))
$$
where $H_*(\Gamma;\mathbb{Z})$ are the singular homology groups of the group $\Gamma$ with coefficients in $\mathbb{Z}$.
Given a group homology class $h\in H_n(\Gamma;\mathbb{Z})$, consider the space 
$$\mathscr{C}(h)$$ of boundaryless $n$-dimensional integral currents with compact supports inside $ {S^\infty}/\lambda_\Gamma(\Gamma)$ which represent the homology class $\hat{\pi}(h)\in \mathbf{H}_n( {S^\infty}/\lambda_\Gamma(\Gamma))$; a more careful definition of $\hat{\pi}$ and $\mathscr{C}(h)$ is given in Subsection \ref{def of C}. For simplicity, we will sometimes call these currents ``cycles representing $h$''. Recall that the notion of mass $\mathbf{M}$ for an integral current is reviewed in Section \ref{appendix a}.
We define the spherical volume of a group homology class as follows:
\begin{definition} [Spherical volume] \label{seconde def}
Let $h\in H_n(\Gamma;\mathbb{Z})$. The \emph{spherical volume} of $h$ is defined as
$$\spherevol(h) = \inf\{\mathbf{M}(C); \quad C\in \mathscr{C}(h)\}.$$
\end{definition} 
This is a homological generalization of the spherical volume first introduced by Besson-Courtois-Gallot in the Riemannian setting \cite[Section 3,I]{BCG91} (see also Subsection \ref{original def}).  The reader can learn in the survey of Kotschick \cite[Section 2]{Kot11}  that this invariant is closely related to a plethora of other fundamental invariants like the simplicial volume, minimal volume, minimal volume entropy...
Besides, this invariant can be computed in many special cases, as will be  explained in Sections \ref{section:uniqueness} and \ref{section:uniqueness2}.

We can now define  spherical Plateau solutions:
\begin{definition}[Spherical Plateau solution]
We call \emph{spherical Plateau solution} for $h$ any $n$-dimensional integral current space $C_\infty$ which is the limit in the intrinsic flat topology of a sequence $\{C_i\}\subset \mathscr{C}(h)$ such that
$$\lim_{i \to \infty} \mathbf{M}(C_i) = \spherevol(h).$$
\end{definition}
The \emph{spherical Plateau problem} consists of studying spherical Plateau solutions and their relation to the original pair $(\Gamma,h)$. 

Note that spherical Plateau solutions $C_\infty$ for a group homology class $h$ always exist by the compactness theorem of Wenger \cite{Wenger11}\cite[Theorem 4.19]{SW11} which serves as a replacement for the compactness theorem of Federer-Fleming in finite dimensions (it suffices to check that given $C_i$ as in the above definition, their masses as well as their diameters are uniformly bounded). Any spherical Plateau solution $C_\infty$ is an integral current space without boundary, with uniformly bounded diameter. The mass of $C_\infty$ satisfies
$$ \mathbf{M}(C_\infty) \leq \spherevol(h)$$
 by lower semicontinuity of the mass under instrinsic flat convergence \cite[Theorem 4.6]{SW11}. 
 However the converse is unclear:
 \begin{question}[Volume convergence]
 Given a group homology class $h$ and a spherical Plateau solution $C_\infty$ for $h$, is it always true that $$ \mathbf{M}(C_\infty) = \spherevol(h)?$$ 
 \end{question}

An $n$-dimensional integral current without boundary $S$ in a complete metric space $(E,d)$ is called mass-minimizing if for any $(n+1)$-dimensional integral current $D$ supported in $(E,d)$, $\mathbf{M}(S)\leq \mathbf{M}(S+\partial D)$. 
In striking contrast with the finite dimensional compact case, there is no non-trivial mass-minimizing integral current without boundary inside the manifold $ {S^\infty}/\lambda_\Gamma(\Gamma)$ when $\Gamma$ is torsion-free. In particular, a spherical Plateau solution $C_\infty$ for $h$ is in general not isometrically embedded inside $ {S^\infty}/\lambda_\Gamma(\Gamma)$ as a cycle in $\mathscr{C}(h)$. That fact follows from the existence of a distance decreasing flow on the subset of nonnegative functions of $ {S^\infty}/\lambda_\Gamma(\Gamma)$, as explained in Remark \ref{not achieved}. This raises the following:
 \begin{question}[Variational structure]
Does any spherical Plateau solution isometrically embed as a mass-minimizing integral current in a quotient of a Hilbert  unit sphere?
 \end{question}

Instead of focusing on a fixed spherical Plateau solution, for a given dimension, one can instead look at the set of all spherical volumes of group homology classes, and the space of all spherical Plateau solutions. In Section \ref{section:uniqueness3}, we will establish the existence of accumulation phenomena for spherical Plateau solutions in all dimensions at least $4$.

\begin{remarque}
All the above definitions for the spherical Plateau problem can be extended or modified to study general orthogonal representations different from the regular representation.
\end{remarque}




\subsection{More details on the set of cycles $\mathscr{C}(h)$.} \label{def of C}

Consider a countable group $\Gamma$. Let us say more about the natural map 
\begin{equation}\label{morphism}
\hat{\pi}: H_*(\Gamma;\mathbb{Z}) \to \mathbf{H}_*( {S^\infty}/\lambda_\Gamma(\Gamma)).
\end{equation}
If $\Gamma$ is finite, there is a point in the finite dimensional sphere $ {S^\infty}$ which is fixed by the whole group $\Gamma$ so it is natural to define $\hat{\pi}$ to be the trivial map sending everything to $\{0\}$.
Assume that $\Gamma$ is infinite. In that case, set
 $$   {S^{\infty,*}} :=\{ x\in  {S^\infty};\quad \text{there is no $g\neq 1$ such that $\lambda_\Gamma(g)x=x$} \}.$$
Then one can check that the restriction of the action of $\Gamma$ on $   {S^{\infty,*}}$ is proper free, $   {S^{\infty,*}}$ is still weakly contractible (i.e. all the higher homotopy groups are trivial) and $   {S^{\infty,*}}/\lambda_\Gamma(\Gamma)$ is a non-complete, infinite-dimensional, Hilbert Riemannian manifold and a classifying space for $\Gamma$.
Since $   {S^{\infty,*}}/\lambda_\Gamma(\Gamma)$ is an open Riemannian manifold, it is well-known that $\mathbf{H}_*(  {S^{\infty,*}}/\lambda_\Gamma(\Gamma))$ is isomorphic to the singular homology groups $H_*(  {S^{\infty,*}}/\lambda_\Gamma(\Gamma);\mathbb{Z})$ \cite{RS09}.
Recall that for any classifying space $B\Gamma$, by definition $$H_*(\Gamma;\mathbb{Z}) = H_*(B\Gamma;\mathbb{Z}).$$
Then we set 
$$\hat{\pi} : H_*(\Gamma;\mathbb{Z}) = H_*(  {S^{\infty,*}}/\lambda_\Gamma(\Gamma);\mathbb{Z}) =\mathbf{H}_*(  {S^{\infty,*}}/\lambda_\Gamma(\Gamma)) \to \mathbf{H}_*( {S^\infty}/\lambda_\Gamma(\Gamma))$$
where the last arrow is induced by the inclusion map $  {S^{\infty,*}}/\lambda_\Gamma(\Gamma) \subset  {S^\infty}/\lambda_\Gamma(\Gamma)$.
When $\Gamma$ is torsion-free infinite, $\hat{\pi}$ is an isomorphism because $  {S^{\infty,*}}/\lambda_\Gamma(\Gamma) =  {S^\infty}/\lambda_\Gamma(\Gamma)$.

Let $h\in H_n(\Gamma;\mathbb{Z})$. Our definition of the family of cycles $\mathscr{C}(h)$ is then the following: 
\begin{align*}
\mathscr{C}(h) := \{& \text{integral currents with compact support in $ {S^\infty}/\lambda_\Gamma(\Gamma)$}\\
&\text{representing the class $\hat{\pi}(h)\in \mathbf{H}_n( {S^\infty}/\lambda_\Gamma(\Gamma))$}\}.
\end{align*}

\subsection{Basic examples}

\subsubsection{Manifolds}
To any smooth closed oriented $n$-manifold $M$ corresponds a spherical Plateau problem.
Indeed, let $\Gamma:=\pi_1(M)$.
By elementary topology, there is a canonical homotopy class of continuous maps from $M$ to a classifying space $B\Gamma$ for $\Gamma$ which induces an isomorphism on fundamental groups. This homotopy class determines a unique homology class 
$$h_M\in H_n(\Gamma;\mathbb{Z}) =  H_n(B\Gamma;\mathbb{Z})$$
called the induced class. Often we will make use of the following notation
$$\spherevol(M):=\spherevol(h_M),$$
and the spherical Plateau solutions for $h_M$ will alternatively be called spherical Plateau solutions for $M$.

One could hope that the area-minimization process in the spherical Plateau problem ``simplifies'' the topology of $M$.
\begin{question}[Complexity]
How does the topological complexity of a closed oriented manifold $M$ compares with that of its spherical Plateau solutions?
 \end{question}

\begin{question}[Stabilization]
Consider a closed oriented manifold $M$ and its induced homology class $h_M\in H_n(\Gamma;\mathbb{Z})$. 
How do $\spherevol(m h_M)$ and its corresponding spherical Plateau solutions behave with respect to the integer $m$? What are the properties of $\lim_{m\to \infty} \frac{1}{m} \spherevol(m h_M)$? 
 \end{question}
 
For related results about multiples of homology classes in finite-dimensional Riemannian manifolds, see \cite{Federer74} \cite{Lawson75} \cite{Morgan84} \cite{White84} \cite{Liu23}. For progress on similar questions for the volume entropy, see \cite{BS23} and references therein.

\subsubsection{Amenable groups}
Suppose $\Gamma$ is finite and let $h\in H_n(\Gamma;\mathbb{Z})$.
Then by definition of $\mathscr{C}(h)$ (see Subsection \ref{def of C}), the zero integral current in $ {S^\infty}/\lambda_\Gamma(\Gamma)$ is in $\mathscr{C}(h)$, so $\spherevol(h)=0$.

If $\Gamma=\mathbb{Z}$ and $h\in H_1(\Gamma;\mathbb{Z})$ (for instance if $M=S^1$ and $h=h_M$), then $\spherevol(h)=0$ and the only spherical Plateau solution is the zero integral current space. To see this, one can represent $h$ by a disjoint union of embedded oriented loops $c_1,...,c_L \subset   {S^{\infty}}/\lambda_\Gamma(\Gamma)$,
where each $c_i$ is the projection of some segment in $   {S^{\infty}}$ joining an element $f$ to $\gamma_i .f$ where $\gamma_i \in \Gamma$ depends on $c_i$.  Now one can separately move each $c_i$ homotopically to a curve of arbitrarily small length, since for each $i$ there is $f_i$ such that $\|f_i - \gamma_i.f_i\|_{\ell^2}$ is arbitrarily small. 

In fact, these arguments can be generalized to any amenable groups thanks to the so-called Dixmier condition: a finitely generated group $\Gamma$ is \emph{amenable} if and only if for any $\epsilon>0$ and any finite subset $S\subset \Gamma$, there is $u\in  {S^\infty}$ such that $\|s.u-u\|_{\ell^2}\leq \epsilon$ for all $s\in S$.

 In particular it is possible to show that if $\Gamma$ is amenable, then $\spherevol(h)=0$ for all $h\in H_*(\Gamma;\mathbb{Z})$.
With a bit more work, one can work out a vanishing theorem analogous to the vanishing theorem for the simplicial volume due to Gromov  \cite[Section 3.1]{Gromov82}.
In that sense the spherical Plateau problem, like many geometric invariants (simplicial volume, volume entropy etc.), is only sensitive to ``large'' groups.
\begin{question}[Vanishing of spherical volume]
What are useful characterizations of group homology classes with vanishing spherical volume? \end{question}

\subsection{Some distance non-increasing maps} \label{some distance non increasing maps}
In this subsection, we collect some useful distance non-increasing maps.

\subsubsection{Group homomorphisms}

Let $F,G$ be countable groups with regular representations $\lambda_F,\lambda_G$, and a homomorphism $\theta: F\to G$.
The homomorphism $\theta$ induces a map sending $u\in \ell^2(F)$ to $\Theta(u)\in \ell^2(G)$ defined by
\begin{equation} \label{homomdecrease}
\Theta(u)(y) := 
\left\{ \begin{array}{rcl}
\big( \sum_{x\in \theta^{-1}(y)} |u(x)|^2 \big)^{1/2} & \mbox{if} & y\in \theta(F) \\ 
0 & \mbox{if} & y\notin \theta(F)
\end{array}\right..
\end{equation}
By Cauchy-Schwarz, this map is distance non-increasing.
One checks directly that this map is also $F$-equivariant, which implies that there is an induced map 
$$\Theta :  {S^\infty}/\lambda_F(F) \to  {S^\infty}/\lambda_G(G)$$ which is still distance non-increasing.

Now let $h\in H_n(F;\mathbb{Z})$ and let $\theta_*(h) \in H_n(G;\mathbb{Z})$ be the class equal to the push-forward of $h$ by $\theta$. 
From the above, one can show that:
\begin{equation}  \label{homomorphism}
\spherevol(h) \geq  \spherevol(\theta_*(h)).
\end{equation}

\subsubsection{Absolute value} \label{subsub absolute}

Let $\Gamma$ be countable, $h\in H_n(\Gamma;\mathbb{Z})$. 
By a similar argument using Cauchy-Schwarz, note that there is a distance non-increasing equivariant map 
$$\mathcal{A} :  {S^\infty}\to  {S^\infty}$$
$$u\mapsto |u|,$$
which descends to a distance non-increasing map 
\begin{equation}\label{overlinepsi}
\mathcal{A} : {S^\infty}/\lambda_\Gamma(\Gamma)\to  {S^\infty}/\lambda_\Gamma(\Gamma).
\end{equation}

\subsubsection{Spherical convolution}

Suppose here that $\Gamma$ is countably infinite and has no torsion.
Let  $\tilde{\eta}:\Gamma\to [0,1]$  be a function with $\sum_\gamma \tilde{\eta}(\gamma) =1$,
and suppose also that $\tilde{\eta}$ is strictly positive everywhere.  The ``spherical convolution'' of $f_0$ by $\tilde{\eta}$ is defined as
$$\tilde{\eta}\star f_0(\gamma) := \big[\sum_{\gamma'} |f_0(\gamma')|^2 \tilde{\eta}(\gamma'^{-1}\gamma)\big]^{1/2}.$$
Whenever $f_0\not\equiv 0$, $\tilde{\eta}\star f_0$ has support the whole group $\Gamma$. 
By Cauchy-Schwarz,
\begin{equation} \label{spherical convol}
\star_{\tilde{\eta}} : f_0\mapsto \tilde{\eta}\star f_0
\end{equation}
is a $\Gamma$-equivariant distance non-increasing map which preserves the $\ell^2$ norm of functions (essentially this is the discrete version of \cite[Remarque 2.7]{BCG95}).
Hence $\star_{\tilde{\eta}}$ induces a well-defined distance non-increasing map from $ {S^\infty}/\lambda_\Gamma(\Gamma)$ to itself, homotopic to the identity. Importantly, one checks that $\star_{\tilde{\eta}} $ is differentiable on $ {S^\infty}$, and its differential strictly decreases the norm of any tangent vector at any point of $ {S^\infty}$. In particular, $\star_{\tilde{\eta}}$  is strictly distance decreasing.

\begin{remarque}\label{not achieved}
\emph{Why is the intrinsic flat topology needed?}
When defining spherical Plateau solutions, one might hope that they are embedded in the spherical quotient $ {S^\infty}/\lambda_\Gamma(\Gamma)$, but this is not possible for the following reason. 
Consider any non-zero integral current $C\in \mathscr{C}(h)$ in the spherical quotient $ {S^\infty}/\lambda_\Gamma(\Gamma)$, where $\Gamma$ is infinite  torsion-free and $h\in H_n(\Gamma;\mathbb{Z})$. By 
applying the spherical convolution $\star_{\tilde{\eta}} $ with $\tilde{\eta}$ as above, we get a current $(\star_{\tilde{\eta}})_\sharp(C)$ which still belongs to $ \mathscr{C}(h)$ but with strictly smaller mass by the area formula reviewed in Subsection \ref{subsection:area}. As a consequence, whenever $h$ is nonzero, no element of $\mathscr{C}(h)$ can actually achieve the equality $\mathbf{M}(C) = \spherevol(h)$. 
More generally, this argument shows that there is no non-trivial mass-minimizing integral current in $ {S^\infty}/\lambda_\Gamma(\Gamma)$!
This explains why notions such as ``intrinsic flat convergence'', ``integral current spaces'' \cite{SW11} or ``ultralimit''  are necessary in formulating the spherical Plateau problem.
\end{remarque}

\subsection{Besson-Courtois-Gallot's definition of the spherical volume} \label{original def}

There are several natural ways to define the setup for the spherical Plateau problem. Conjecturally, many of those setups lead to \emph{equivalent} notions of spherical volume and spherical Plateau solutions, whenever they make sense.
In this subsection, we recall the original setup of Besson-Courtois-Gallot in the case of smooth manifolds, which inspired the definition of the spherical Plateau problem.

Let $(M,g)$ be a closed, oriented, smooth Riemannian $n$-manifold with fundamental group $\Gamma:= \pi_1(M)$, let $(\tilde{M},g)$ be its universal cover.
Let  $D_M\subset \tilde{M}$ be a Borel fundamental domain and let $S_2(\tilde{M},g)$ be the unit sphere of $L^2(\tilde{M},g)$, endowed with the standard metric $\mathbf{g}_{L^2(\tilde{M},g)}$. 
There is a natural action $\lambda_{(\tilde{M},g)}$ of $\Gamma$ by isometries on $S_2(\tilde{M},g)$, and changing the metric $g$ yields $\Gamma$-equivariantly isometric spaces, see \cite[Section 3, I]{BCG91}. 
Besson-Courtois-Gallot  \cite[Section 3, I]{BCG91} defined the spherical volume of $M$ as
\begin{align*}
&\spherevol_{BCG}(M)  \\ &:=  \inf\{ \Vol(D_M,\phi^*\mathbf{g}_{L^2(\tilde{M},g)}); \quad \text{ $\phi:\tilde{M}\to S_2(\tilde{M},g)$ is a $\Gamma$-equivariant immersion}\}.
\end{align*} 
In fact, we could have used $S^\infty\subset \ell^2(\Gamma)$,  the regular representation $\lambda_{\Gamma}$, and $\Gamma$-equivariant smooth maps, instead of respectively $S_2(\tilde{M},g) \subset L^2(\tilde{M},g)$, $\lambda_{(\tilde{M},g)}$, and $\Gamma$-equivariant immersions, without changing the value of $\spherevol_{BCG}(M)$.
 
By \cite{BCG91}, the simplicial volume \cite{Gromov82} of a closed oriented $n$-manifold $M$ is related to the spherical volume of Besson-Courtois-Gallot by
$$\|M\| \leq C_n \spherevol_{BCG}(M).$$

It is not hard to see that 
$$\spherevol_{BCG}(M) \geq \spherevol(M).$$
If the fundamental group $\Gamma$ is torsion-free,
it is possible to show that \emph{equality} holds:
this  non-trivial result is essentially contained in \cite[Lemma 3.10]{Babenko06}\cite[Section 2]{Brunnbauer08} when the dimension of $M$ is at least $3$. 
\begin{question}[Equivalence of definitions]
Do we always have 
$$\spherevol_{BCG}(M) = \spherevol(M)?$$
\end{question}

\subsection{Intrinsic isomorphism} \label{Plateau topology}

It is sometimes natural and helpful to consider spherical Plateau solutions up to ``intrinsic isomorphism'', and discard some of the non-infinitesimal information. Indeed under such an equivalence relation, uniqueness and rigidity properties emerge naturally, as we will see in Sections \ref{section:uniqueness}, \ref{section:uniqueness2} and \ref{section:uniqueness3}. 
Given a metric space $(X,d)$, the intrinsic metric on $X$ induced by $d$ is denoted by $L_{d}$ \cite[Chapter 2, Section 2.3]{BBI22}. By convention, the $L_d$-distance between points of different path connected components of $(X,d)$ is $\infty$.   
Note that the identity map  
$$\id: (X,L_d)\to (X,d)$$ is always $1$-Lipschitz. 

Consider an integral current space $C=(X,d,T)$ and an oriented complete finite volume Riemannian manifold $(N,g_N)$ which is not necessarily connected, which induces the integral current space $(N,\dist_{g_N},\llbracket 1_N\rrbracket)$.
\begin{definition} \label{definition:intrinsic isomorphism}
We say that $C=(X,d,T)$ is \emph{intrinsically isomorphic} to $(N,g_N)$.   
if there is an isometry 
$$\varphi: (N,\dist_{g_N}) \to (X,L_d)$$ 
such that 
$$(\id\circ \varphi)_\sharp \llbracket 1_N\rrbracket = T.$$
\end{definition}
One could also try to formulate more general definitions of the type ``Two integral current spaces are intrinsically  isomorphic if...''.
For clarity, we emphasize that ``intrinsically isomorphic'' does not imply ``at intrinsic flat distance $0$ from each other''.


 \section{Spherical Plateau solutions for hyperbolic manifolds} \label{section:uniqueness}

In this section, we review how Besson-Courtois-Gallot compute the spherical volume of hyperbolic manifolds and how it is used in their proof of the volume entropy inequality. 
Then we outline a proof of our first main theorem, the uniqueness of spherical Plateau solutions for hyperbolic manifolds, up to intrinsic isomorphism. This result in turn leads to an area rigidity property for the regular representation of fundamental groups of hyperbolic manifolds.
All those discussions  can be adapted to the general rank one locally symmetric case.
 
\subsection{The spherical volume of hyperbolic manifolds} 

The spherical volume of closed oriented hyperbolic manifolds was computed by Besson-Courtois-Gallot.
For completeness, we outline their proof in our setting when the dimension is at least $3$.
\begin{theo} \cite{BCG95,BCG96} \label{theorem:bcg}
Let $(M,g_0)$ be a closed oriented hyperbolic manifold of dimension at least $3$. 
Then 
\begin{equation} \label{bcg}
\spherevol(M) = \Vol(M,\frac{(n-1)^2}{{4n}}g_0).
\end{equation}
\end{theo}

\begin{proof}
Let us sketch the proof, which is due to Besson-Courtois-Gallot \cite[Sections 5 and 6]{BCG95}. Suppose that the dimension of $M$ is at least $3$, let $\Gamma:=\pi_1(M)$ and let $h_M\in H_n(\Gamma;\mathbb{Z})$ be the induced homology class.

Let $C\in \mathscr{C}(h_M)$ be a cycle in $ {S^\infty}/\lambda_\Gamma(\Gamma)$ representing $h_M$.
It is convenient to consider cycles which are polyhedral chains. Recall from Lemma \ref{approx polyh} that $C$ can be nicely approximated by a polyhedral chain  $P\in \mathscr{C}(h_M)$. Moreover, by a further perturbation if necessary, we get another polyhedral chain $P'\in \mathscr{C}(h_M)$ so that $\spt(P') \subset \mathbb{S}^+/\lambda_\Gamma(\Gamma)$ where $\mathbb{S}^+$ is the set of functions with finite support (see Section \ref{appendix b}). Afterwards, we will assume without loss of generality that $C$ is such a polyhedral chain.

Fix a point $o\in \tilde{M}$ in the universal cover of $M$, and let $\mathbf{p}\in M$ be its projection in $M$. Let 
$$\mathrm{Bar}: \mathbb{S}^+/\lambda_\Gamma(\Gamma) \to M$$
be the barycenter map whose definition is given in Section \ref{appendix b}. 
Recall that the restriction of $\mathrm{Bar}$ to $C$ is a Lipschitz map.
By $\Gamma$-equivariance of $\mathrm{Bar}: \mathbb{S}^+ \to \tilde{M}$, 
\begin{equation} \label{1_M}
\mathrm{Bar}_\sharp(C) = \llbracket 1_M \rrbracket.
\end{equation}

For almost every point $q\in \spt(C)$, $C$ admits a tangent $n$-plane at $q$.
The $n$-dimensional Jacobian of ${\mathrm{Bar}}$ along the tangent $n$-plane is well-defined and is bounded from above by $\big(\frac{4n}{(n-1)^2}\big)^{n/2}$ thanks to  (\ref{jacbar<}) in Lemma \ref{astuce}. This implies by (\ref{1_M}) and the area formula in Subsection \ref{subsection:area} that
$$\mathbf{M}(C) \geq    \big(\frac{(n-1)^2}{{4n}}\big)^{n/2} \Vol(M,g_0).$$
Since $C\in \mathscr{C}(h_M)$ has mass arbitrary close to $\spherevol(M)$, we conclude that
$$\spherevol(M) \geq  \Vol(M,\frac{(n-1)^2}{{4n}}g_0).$$

The inverse inequality 
$$\spherevol(M) \leq  \Vol(M,\frac{(n-1)^2}{{4n}}g_0)$$
directly follows from a general inequality between the spherical volume and the volume entropy of a closed Riemannian manifold, see \cite[Corollary 3.13]{BCG91}. It states that if $(M,g)$ is a closed oriented Riemannian $n$-manifold, then 
\begin{equation}\label{bcgse}
\spherevol(M) \leq \Vol(M,\frac{h(g)^2}{4n}g)
\end{equation}
where $h(g)$ is the volume entropy of the Riemannian metric $g$ (the definition is recalled in (\ref{def entropy})). Note that in \cite[Section 3, I]{BCG91}, $\spherevol(M)$ is denoted $T(M)$. 
The proof of (\ref{bcgse}) is based on the following maps.
Denote by $D_M$ a Borel fundamental domain in the universal cover $\tilde{M}$ for the action of $\Gamma$ and let $\gamma.D_M$ be its image by an element $\gamma\in \Gamma$.  For $c>h(g)$, defined the equivariant map
\begin{align}\label{poissonc}
\begin{split}
\mathcal{P}_c &:  \tilde{M} \to   {S^\infty}\subset \ell^2(\Gamma) \\
 & x\mapsto \{\gamma \mapsto \frac{1}{\|e^{-\frac{c}{2}\dist_{g}(x,.)}\|_{L^2(\tilde{M},g)}}   \big[\int_{\gamma. D_M}  e^{-c\dist_{g}(x,u)} \dvol_g(u)\big]^{1/2}\}.
 \end{split}
 \end{align}
Properties of $\mathcal{P}_c$  \cite[Proof of Lemma 3.1]{BCG91} \cite[Lemma 3.1]{Antoine23b} imply that 
$$ \liminf_{c\to h(g)} \mathbf{M}((\mathcal{P}_c)_\sharp \llbracket 1_M\rrbracket) \leq \frac{h(g)^n}{2^n n^{n/2}} \Vol(M,g).$$
Since by equivariance, $(\mathcal{P}_c)_\sharp \llbracket 1_M\rrbracket $ belongs to $\mathscr{C}(h_M)$, the above inequality shows (\ref{bcgse}).


\end{proof}

The analogue of Theorem \ref{theorem:bcg} for closed oriented surfaces was shown in \cite[Proposition 3.9]{BCG91} by methods specific to the $2$-dimensional case.
Theorem \ref{theorem:bcg} also extends directly to rank one (i.e. negatively curved) locally symmetric manifolds \cite{BCG95,BCG96,Ruan22}. 
As for higher rank locally symmetric manifolds, following the logic of Besson-Courtois-Gallot \cite{BCG95} (see Subsection \ref{and back again}), a big step towards the entropy rigidity conjecture in higher rank  \cite[Question (5)]{BCG96}\cite{CF02} would be the computation of their spherical volumes:
\begin{question}[Spherical volume in higher rank]
What is the spherical volume of a closed oriented locally symmetric manifold of higher rank?
\end{question}

\subsection{From spherical volume to volume entropy and back again} \label{and back again}
Given  a closed manifold $M$, if $g$ is a Riemannian metric on $M$, let $h(g)$ denote its volume entropy:
\begin{equation}\label{def entropy}
h(g):=\lim_{R\to \infty} \frac{\log \Vol(\tilde{B}_g(o,R))}{R}
\end{equation}
where $(\tilde{B}_g(o,R))$ denotes the geodesic $R$-ball centered at some point $o$ in the universal cover $(\tilde{M},g)$ of $(M,g)$. 
The fundamental volume entropy inequality for hyperbolic manifolds \cite{BCG95,BCG96} states that if $(M,g_0)$ is a hyperbolic manifold of dimension $n\geq 2$, and if $g$ is any other metric on $M$, then
$$h(g)^n\Vol(M,g)\geq h(g_0)^n\Vol(M,g_0).$$ 
Equivalently, if $g$ is normalized to have volume entropy $h(g_0)$, then
\begin{equation}\label{volvol}
\Vol(M,g)\geq \Vol(M,g_0).
\end{equation}
Moreover when $n\geq 3$, equality holds exactly when $g$ is isometric to $g_0$. A slight extension of this theorem implies the classical Mostow's rigidity theorem \cite{BCG95}.

Nowadays, the proof of Besson-Courtois-Gallot \cite{BCG95,BCG96} in the case $n\geq3$ is primarily remembered for their barycenter map. Yet, a pivotal insight in the original paper \cite{BCG95} is that determining the minimal volume entropy of hyperbolic manifolds can be reduced to the computation of their spherical volume.
This aspect of their work is rooted in the theory of minimal surfaces and calibrations, and should be recalled.  We can summarize their strategy as follows. Let $M$ be as before and let $\Gamma:=\pi_1(M)$. If $(M,g)$ is normalized so that $h(g)=h(g_0)=n-1$, then by using properties of the maps $\mathcal{P}_c$ defined in (\ref{poissonc}) for $c>h(g_0)$  \cite[Proof of Lemma 3.1]{BCG91} \cite[Lemma 3.1]{Antoine23b}, the image $\mathcal{P}_c(M)$ inside $ {S^\infty}/\lambda_\Gamma(\Gamma)$ has volume at most $\frac{c^n}{2^nn^{n/2}} \Vol(M,g)$. But since $\mathcal{P}_c(M)$ determines a cycle in $\mathscr{C}(h_M)$, by Theorem \ref{theorem:bcg}, the volume of $\mathcal{P}_c(M)$ is at least the spherical volume $\frac{(n-1)^n}{2^nn^{n/2}} \Vol(M,g_0)$. As $c$ can be taken arbitrarily close to $n-1$, we find that (\ref{volvol}) is true. The rigidity part is shown as follows: if equality holds, then as $c\to n-1$, the composition $\mathrm{Bar}\circ \mathcal{P}_c$ is almost a Riemannian isometry and converges to a Riemannian isometry, so in the limit we conclude that $g$ is isometric to $g_0$.

Actually, the summary above deviates a bit from the original presentation of Besson-Courtois-Gallot: in \cite{BCG95}, the authors use a spherical quotient different from the quotient $ {S^\infty}/\lambda_\Gamma(\Gamma)$ that we choose to work with, and they use an explicit minimal isometric embedding of $(M,\frac{(n-1)^2}{4n}g_0)$ into that spherical quotient. Let us describe those geometric objects. 
Consider the universal cover $(\tilde{M},g_0)$ with the hyperbolic metric. Fix a basepoint $\mathbf{o}\in \tilde{M}$, let $\partial \tilde{M}$ be the boundary at infinity of $\tilde{M}$ with the standard probability measure determined by $\mathbf{o}$, and let $S_2(\partial \tilde{M})$ be the unit sphere in $L^2(\partial \tilde{M})$.  For any $\theta\in \partial \tilde{M}$, the corresponding Busemann function is defined for any $x\in \tilde{M}$ as
$$B_\theta(x) := \lim_{t\to \infty} (\dist_{g_0}(y,c(t)) -t)$$  
where $c:[0,\infty)$ is the half-geodesic starting at $\mathbf{o}$, and converging to $\theta$. 
The group $\Gamma$ acts naturally on $\tilde{M}$ and $\partial \tilde{M}$. There is also a natural proper, free, isometric action $\rho_B$ of $\Gamma$ on $S_2(\partial \tilde{M})$ given, for all $\gamma\in \Gamma$, $f\in S_2(\partial \tilde{M})$, by
\begin{equation}\label{boundary repa}
\rho_B(\gamma). f (\theta) = f(\gamma^{-1}(\theta)) e^{-\frac{n-1}{2}B_\theta(\gamma(\mathbf{o}))},
\end{equation}
see \cite[Lemme 2.2]{BCG95}. 
This action $\rho_B$ is also called a boundary representation of $\Gamma$. It is possible to show that the spherical quotient $S_2(\partial \tilde{M})/\rho_B(\Gamma)$ is not isometric to our spherical quotient $ {S^\infty}/\lambda_\Gamma(\Gamma)$ defined with the regular representation of $\Gamma$. Nevertheless, these two spaces are closely related, as $S_2(\partial \tilde{M})/\rho_B(\Gamma)$ is contained in the ``ultralimit of $ {S^\infty}/\lambda_\Gamma(\Gamma)$''. 

Consider the following embedding
\begin{equation}\label{special emb p}
\begin{split}
\mathscr{P} : (\tilde{M},\frac{(n-1)^2}{4n}g_0) \to S_2(\partial \tilde{M})\\
\mathscr{P}(x) := \{\theta \mapsto e^{-\frac{n-1}{2}B_\theta(x)}\}.
\end{split}
\end{equation}
It is not too hard to check that $\mathscr{P} $ is an isometric and minimal embedding.
As explained in \cite[Section 2]{BCG95}, this map is equivariant, and so after taking the quotient by $\Gamma$ and rescaling the metric $g_0$, we get a map
$$\mathscr{P} : (M,\frac{(n-1)^2}{4n}g_0) \to S_2(\partial \tilde{M})/\rho_B(\Gamma).$$
Again, this $\mathscr{P} $ is an isometric and minimal embedding. Strikingly, Besson-Courtois-Gallot discovered \cite[Proposition 5.7, Section 6]{BCG95} that  $\mathscr{P}(M)$ is an area-minimizing $n$-submanifold of  the Hilbert Riemannian manifold $S_2(\partial \tilde{M})/\rho_B(\Gamma)$, by arguing that $\mathscr{P}(M)$ is in fact a calibrated submanifold. 
When $n\geq 3$, the calibration is essentially constructed as the pull-back of the volume form on $(M,g_0)$ by a barycenter map   $$\mathrm{Bar}:S_2(\partial \tilde{M})/\rho_B(\Gamma)\to (M,g_0).$$

Now, here is the relevance of $\mathscr{P}(M)$ and our discussion about  the spherical Plateau problem. If $g=g_0$ then the image $\mathcal{P}_c(M)$ in $ {S^\infty}/\lambda_\Gamma(\Gamma)$ has volume converging to the spherical volume $\frac{(n-1)^n}{2^nn^{n/2}} \Vol(M,g_0)$ as $c\to n-1$. It can be shown that, as integral currents, $\mathcal{P}_c(M)$ converges in the intrinsic flat topology to $\mathscr{P}(M)$. With the vocabulary of Section \ref{definition spp}, this means that there is an explicit minimizing sequence of cycles in $\mathscr{C}(h_M)$ inside $ {S^\infty}/\lambda_\Gamma(\Gamma)$ converging to a spherical Plateau solution which is given by $\mathscr{P}(M)$. 
Is $\mathscr{P}(M)$ the only spherical Plateau solution up to isomorphism when $n\geq3$ (see Question \ref{question:unique})? We will show in Theorem \ref{uniqueness1} that it is the unique one at least up to instrinsic isomorphism.

\subsection{Intrinsic uniqueness and rigidity of spherical Plateau solutions}

We now come to our first new result, which states that spherical Plateau solutions are unique for hyperbolic manifolds, up to intrinsic isomorphism (see Definition \ref{definition:intrinsic isomorphism}):
\begin{theo} \label{uniqueness1}
Let $(M,g_0)$ be a closed oriented hyperbolic manifold of dimension $n\geq 3$. 
 Then any spherical Plateau solution for $M$ is intrinsically isomorphic to $(M,\frac{(n-1)^2}{4n}g_0)$.
\end{theo}

\begin{proof}[Outline of proof]
Let $\Gamma:=\pi_1(M)$, and let $h_M\in H_n(\Gamma;\mathbb{Z})$ be the fundamental homology class, where $n\geq 3$ is the dimension of $M$.
Let $C_i\in \mathscr{C}(h_M)$ be a minimizing sequence, namely
\begin{equation} \label{minim}
\lim\mathbf{M}(C_i) = \spherevol(M) = \Vol(M,\frac{(n-1)^2}{4n}g_0).
\end{equation}
The second equality above is Theorem \ref{bcg}.
By Lemma \ref{approx polyh}, as in the proof of Theorem \ref{bcg}, we can assume $C_i$ to be polyhedral chains, with support in $\mathbb{S}^+/\lambda_\Gamma(\Gamma)$, where $\mathbb{S}^+$ is the set of functions with finite support.

We assume that $C_i$ converges to a spherical Plateau solution 
$$C_\infty=(X_\infty,d_\infty,S_\infty).$$
We use the notation
$$g':=\frac{(n-1)^2}{4n}g_0.$$
Jacobians, lengths and distances will be computed with respect to $g'$. Fix $o\in \tilde{M}$ a point in the universal cover.
Let 
$$\mathrm{Bar} : \spt(C_i) \subset \mathbb{S}^+/\lambda_\Gamma(\Gamma)\to M$$
be the barycenter map defined in Section \ref{appendix b}. 

The Jacobian bound (\ref{jacbar<}) on the barycenter map in Lemma \ref{astuce}, (\ref{1_M}) and (\ref{minim}) imply that, as $i\to \infty$, on a larger and larger region of the polyhedral chain $\spt(C_i)$, the Jacobian bound is almost tight. In particular there is a ``good region'' $\Omega_i\subset \spt(C_i)$ such that the mass of $C_i\llcorner \Omega_i$ converges to $\spherevol(M)$, the Jacobian of $\mathrm{Bar}$ converges to $1$ on $\Omega_i$ (with respect to the metric $g'$), and $\mathrm{Bar}$ is injective on $\Omega_i$.

The barycenter map $\mathrm{Bar} : \spt(C_i) \to M$ is not uniformly Lipschitz as $i\to \infty$. Nevertheless, by Lemma \ref{astuce}, it is indeed uniformly Lipschitz in small neighborhoods of $\Omega_i$. Besides, it can be shown using Lemma \ref{astuce} that the barycenter map is almost $1$-Lispchitz in small balls near $\Omega_i$.

Take $D_i$ the restriction of $C_i$ to an $\tilde{r}$-neighborhood of the good region $\Omega_i\subset \spt(C_i)$, for some small fixed $\tilde{r}>0$. By construction, $D_i$ still converges in the intrinsic flat topology to $C_\infty$. The upshot is that the mass of the boundary $\partial D_i$ converges to $0$, the barycenter map restricted to $\spt(D_i)$ has a uniform Lipschitz bound and is almost $1$-Lispchitz in small balls near $\Omega_i$.
Those properties will be important in the next steps.

By the previous paragraph, one can embed the sequence of integral currents $D_i$ in a Banach space where they converge in the flat topology to $C_\infty=(X_\infty,d_\infty,S_\infty)$ viewed as an integral current in that Banach space. The barycenter map $\mathrm{Bar}$ is well-defined on each $\spt(D_i)$ so by an Arzel\'{a}-Ascoli type argument and the uniform Lipschitz control \cite[Theorem 6.1]{Sor18}, there is a natural Lipschitz ``limit barycenter map'' 
$$\mathrm{Bar}_\infty :\spt S_\infty \to M$$ where the support $\spt S_\infty$ coincides with the completion of $(X_\infty,d_\infty)$. Moreover, by slightly extending  \cite[Lemma 7.3]{GCS23}, it can be shown that 
\begin{equation}\label{coherent}
(\mathrm{Bar}_\infty)_\sharp S_\infty= \llbracket 1_M\rrbracket.
\end{equation}

By Lemma \ref{astuce}, when the Jacobian bound for the barycenter map is almost saturated, the differential of $\mathrm{Bar}$ is almost a linear isometry. Since the Jacobian bound for $\mathrm{Bar}$ indeed becomes arbitrarily close to the sharp upper bound on $\Omega_i\subset \spt(D_i)$ as $i\to \infty$, $\mathrm{Bar}$ is almost a Riemannian isometry for large $i$.
So intuitively, we expect that the differential of the limit barycenter map $\mathrm{Bar}_\infty$ is exactly a linear isometry at any point of $\spt S_\infty$, namely that $\mathrm{Bar}_\infty$ is a Riemannian isometry. This would essentially finish the proof. That sketch would work if the convergence of $C_i$ to $C_\infty$ was known to be smooth (at least outside of a small singular set) \cite{BBCG12}, or if the limit map was known to be $1$-Lipschitz (see for instance for such Lipschitz-volume rigidity results \cite[Proposition C.1]{BCG95}, \cite[Sections 3, 4, 5]{BBCG12}, \cite[Theorem 1.1]{DNP23}, \cite[Theorem 1.1]{GCS23} and \cite[Theorem 1.2]{Zus23}). However in our situation $C_\infty=(X_\infty,d_\infty,S_\infty)$ is a priori just an integral current to which $C_i$ converges weakly, and the limit map is never $1$-Lipschitz in our case (even though it will be shown a posteriori to be $1$-Lipschitz with respect to the intrinsic metrics).

In order to show that 
$$\mathrm{Bar}_\infty : \spt S_\infty \to (M,g')$$
is an isometry for the respective intrinsic metrics, we work with the maps $\mathrm{Bar} : \spt(D_i) \to (M,g')$ instead of directly with the limit map. We prove a general result \cite[Proposition 1.4]{Antoine23b} which roughly says that limits of maps $\varphi_i$ which are uniformly Lipschitz, almost Riemannian isometries, and almost $1$-Lipschitz in small balls, are indeed Riemannian isometries.  This general result can be directly applied to the sequence $\mathrm{Bar} : \spt(D_i) \to (M,g')$. The point of this result is that while the limit map $\mathrm{Bar}_\infty$ is constructed using the extrinsic metrics on $\spt(D_i)$, the intrinsic information about  $\mathrm{Bar}$ (namely the fact that it is almost a Riemannian isometry) still passes to the limit. 

Let us make some comments on the proof of  \cite[Proposition 1.4]{Antoine23b}  in our specific case. Since $\mathrm{Bar}$ is almost $1$-Lipschitz in small balls near $\Omega_i$, it is simple to check that $\mathrm{Bar}_\infty$ is $1$-Lipschitz for the intrinsic metric $L_{d_\infty}$ on $\spt S_\infty$. Note that we cannot apply Lipschitz-volume rigidity results like \cite[Theorem 1.2]{Zus23}, because we do not know if a priori there is an integral current $T$ on the completion of $(\spt S_\infty,L_\infty)$ such that $(\mathrm{Bar}_\infty)_\sharp T = \llbracket 1_M\rrbracket$, and if volumes for the new metric are preserved. Instead, we argue directly as follows.
We want to show that conversely $\mathrm{Bar}_\infty$ does not decrease distances for the intrinsic metrics: given $x,y\in \spt S_\infty$, and a geodesic segment $\sigma$ between $\mathrm{Bar}_\infty(x)$ and $\mathrm{Bar}_\infty(y)$ in $(M,g')$, we want to lift $\sigma$ to a segment of  same length in $\spt S_\infty$. Technically, the main tool is the coarea formula 
reviewed in Subsections \ref{subsection:area} and 
Sard's lemma. As an application, for each $i$, we can perturb a bit $\sigma$ to a nearby curve $\sigma_i$ such that the preimage
$$\varkappa_i :=\mathrm{Bar}^{-1}(\sigma_i) \cap \spt(D_i)$$
is a rectifiable curve contained inside $\spt(D_i)$ 
of length converging to $\length_{g'}(\sigma)$ as $i\to \infty$. Heuristically, this is because the coarea formula ensures that most of $\varkappa_i $ is contained in the good region $\Omega_i$, where the differential of $\mathrm{Bar}$ is almost a linear isometry.
Moreover $\varkappa_i$ is the union of a segment with two endpoints and some closed curves of small lengths, and the segment component converges as $i\to \infty$ to a rectifiable segment inside the support $\spt S_\infty$ with endpoints $x$ and $y$, and length equal to $\length_{g'}(\sigma)$. 

After proving that $\mathrm{Bar}_\infty : \spt S_\infty \to (M,g')$ is an isometry for the respective intrinsic metrics, we can conclude the proof using (\ref{coherent}) and the inverse map $\varphi:=(\mathrm{Bar}_\infty)^{-1}$ from $(M,\dist_{g'})$ to $\spt S_\infty$ in Definition \ref{definition:intrinsic isomorphism}.



\end{proof}

From Theorem \ref{uniqueness1}, we deduce the following \emph{area rigidity} result. Let $(M,g_0)$ be a closed oriented hyperbolic manifold of dimension $n\geq 3$ and set $\Gamma:=\pi_1(M)$. 
A fundamental family of representations of $\Gamma$ is the set of orthogonal representations $\rho:\Gamma\to \End(H)$ which are \emph{weakly equivalent} to the regular representation $\lambda_\Gamma : \Gamma \to \End(\ell^2(\Gamma))$, see \cite[Definition F.1.1]{BDLHV08}.
We say\footnote{Usually, one considers complex Hilbert spaces and unitary representations but in this paper we will only consider real Hilbert spaces and orthogonal representations.} that an orthogonal representation $\rho:\Gamma\to \End(H)$  is weakly equivalent to $\lambda_\Gamma$ if the following holds:  for every $\xi \in H$, every finite subset $Q$ of $\Gamma$, and every $\epsilon>0$, there exist $\eta_1,...,\eta_k$ in $\ell^2(\Gamma)$ such that for all $g\in Q$,
 $$\big| \langle \rho(g)\xi,\xi \rangle - \sum_{j=1}^k \langle \lambda_\Gamma(g)\eta_j,\eta_j\rangle \big| <\epsilon,$$
and vice versa for every $\eta \in \ell^2(\Gamma)$, every finite subset $Q$ of $\Gamma$, and every $\epsilon>0$, there exist $\xi_1,...,\xi_l$ in $H$ such that for all $g\in Q$,
 $$\big| \langle \lambda_\Gamma(g)\eta,\eta \rangle - \sum_{j=1}^l \langle \rho(g)\xi_j,\xi_j\rangle \big| <\epsilon,$$

Now let $\tilde{M}$ be the universal cover of $M$, on which $\Gamma$ acts properly freely, let $g_0$ be the lifted hyperbolic metric on $\tilde{M}$, and let $D_M$ be a Borel fundamental domain in $\tilde{M}$.

\begin{coro}\label{rigidity}
Let $S_H$ be the unit sphere in a Hilbert space $(H,\mathbf{g}_H)$ and let  $\rho:\Gamma\to \End(H)$ be an orthogonal representation weakly equivalent to $\lambda_\Gamma$. Consider a smooth $\Gamma$-equivariant immersion
$$f: \tilde{M}\to S_H.$$
Then 
$$\Vol(D_M,f^*\mathbf{g}_H) \geq \spherevol(M),$$
with equality if and only if
$(f(\tilde{M}),\mathbf{g}_H)$ is an embedded $n$-plane in $S_H$ and is Riemannian isometric to $(\tilde{M},\frac{(n-1)^2}{4n} g_0)$.
\end{coro}

\begin{proof}[Outline of proof]
Consider the infinite direct sum of the regular representation, denoted by
$$\bigoplus^\infty \lambda_\Gamma :\Gamma \to \End(\bigoplus^\infty \ell^2(\Gamma)).$$
Let $(S',\mathbf{g}_{S'})$ be the unit sphere of $\bigoplus^\infty \ell^2(\Gamma)$ with its standard metric.
The natural map  
$$
\mathcal{A}': (a_1,a_2,...) \in S' \mapsto [\sum_{i\geq 1} a_i^2]^{1/2} \in S^\infty\subset \ell^2(\Gamma)
$$
is $\Gamma$-equivariant and  $1$-Lipschitz by the argument of Subsection \ref{subsub absolute}.

The fact that $\rho$ is weakly equivalent to $\lambda_\Gamma$ means that the representation $\rho$ can be arbitrarily well ``approximated'' by $\bigoplus^\infty \lambda_\Gamma$. Hence for any $\epsilon>0$, there exists a smooth $\Gamma$-equivariant map 
$$f_\epsilon : \tilde{M} \to S'$$
which is $\epsilon$-close to $f$ in the $C^1$-topology, in the sense that 
\begin{equation}\label{oplus map}
\|f_\epsilon^* \mathbf{g}_{S'} - f^*\mathbf{g}_H\|_{C^0(D_M)} \leq \epsilon.
\end{equation}

Consider the Lipschitz $\Gamma$-equivariant map 
$$\mathcal{A}' \circ f_\epsilon : \tilde{M} \to S^\infty \subset \ell^2(\Gamma)$$
which after taking the quotient by $\lambda_\Gamma(\Gamma)$, gives a map 
$$M\to S^\infty/\lambda_\Gamma(\Gamma).$$
The image of this map can be identified with a cycle $C\in \mathscr{C}(h_M)$ whose mass satisfies, by (\ref{oplus map}) and the $1$-Lipschitzness of $\mathcal{A}'$:
\begin{equation}\label{mass c compare}
\mathbf{M}(C) \leq \Vol(D_M,f^*\mathbf{g}_H) + c_\epsilon
\end{equation}
for some $c_\epsilon$ converging to $0$ as $\epsilon \to 0$. This establishes the inequality of the statement by definition of $\spherevol(M)$.

As $\epsilon\to \infty$, the cycle $C$ above converges by construction of $f_\epsilon$ to the quotient $f(\tilde{M})/\rho(\Gamma)$ (viewed as an integral current space).
Suppose that equality holds in the statement, namely 
$$\Vol(D_M,f^*\mathbf{g}_H) = \spherevol(M).$$ Then by (\ref{mass c compare}),  $f(\tilde{M})/\rho(\Gamma)$ is  a spherical Plateau solution, so by the uniqueness result, Theorem \ref{uniqueness1}, we conclude that $(f(\tilde{M}),\mathbf{g}_H)$ is an embedded $n$-plane in $S_H$, Riemannian isometric to $(\tilde{M},\frac{(n-1)^2}{4n} g_0)$ as desired.
\end{proof}

Let us emphasize that, by Subsection \ref{and back again}, Corollary \ref{rigidity} formally implies the entropy inequality and rigidity theorem of Besson-Courtois-Gallot \cite{BCG95}.
We may in fact hope for a much stronger \emph{``extrinsic area rigidity''} for $\rho$ in the equality case of Corollary \ref{rigidity}: $\rho$ should contain as a subrepresentation the boundary representation $\rho_B$ defined in (\ref{boundary repa}), and $f(\tilde{M})$ should be equal to the special minimal $n$-plane $\mathscr{P}(\tilde{M})$ defined in (\ref{special emb p}) up to an isometry of the ambient unit sphere. An analogous extrinsic area rigidity property should hold for hyperbolic surfaces, by considering the energy of harmonic maps, instead of the area of $n$-dimensional minimal surfaces. See Questions \ref{question:unique} and \ref{question:unique2} below for a formulation of those conjectures in terms of spherical Plateau solutions.

Details for the proof of Theorem \ref{uniqueness1} appear in \cite{Antoine23b}, where for simplicity we use a slightly different but equivalent definition of spherical Plateau solutions.
All the arguments above extend to the locally symmetric case of rank one, by \cite{BCG95,BCG96,Ruan22}, and so we obtain:
\begin{theo} \label{locsymneg}
If $(M,g_0)$ is a closed oriented locally symmetric manifold of dimension at least $3$, with negative curvature between $-4$ and $-1$, then any spherical Plateau solution for $M$ is intrinsically isomorphic to $(M,\frac{h(g_0)^2}{4n} g_0)$.
\end{theo}

The spherical Plateau solutions in Theorem \ref{locsymneg} are probably unique, i.e. up to isomorphism and not just up to intrinsic isomorphism. The following \emph{Rigidity Conjecture} is perhaps the main question of this paper. Consider $(M,g_0)$,  a closed oriented locally symmetric manifold of rank one of dimension at least $3$ and with fundamental group $\Gamma$.
\begin{question} [Uniqueness] \label{question:unique}
Does $M$ have a unique spherical Plateau solution up to isomorphism? Does extrinsic area rigidity hold for representations $\rho$ weakly equivalent to the regular representation $\lambda_{\Gamma}$?
\end{question}
 
For closed oriented surfaces $\Sigma$ of genus at least $2$, spherical Plateau solutions are non-unique. Indeed each hyperbolic metric on $\Sigma$, after being rescaled by $\frac{1}{8}$, is intrinsically isomorphic to a spherical Plateau solution. Conjecturally, this is essentially the only source of non-uniqueness. 
While the barycenter map is not useful for surfaces, the calibration constructed in \cite[Section 6]{BCG95} hints at a general classification. Let $\Sigma$ be a closed oriented surface of genus at least $2$, with fundamental group $\Gamma$.
 \begin{question}[Classification for surfaces]\label{question:unique2}
Can one classify the spherical Plateau solutions for $\Sigma$?
Does extrinsic area rigidity hold for representations $\rho$ weakly equivalent to the regular representation $\lambda_{\Gamma}$?
\end{question}

  \section{Spherical Plateau solutions for 3-dimensional manifolds} \label{section:uniqueness2}

\subsection{The spherical volume of $3$-manifolds}

By the Geometrization theorem \cite{KL08,MT14}, any closed oriented $3$-manifold $M$ is decomposable into a connected sum of irreducible $3$-manifolds $Y_1,...,Y_m$, each of which is divided into a hyperbolic part and a graph manifold (each part could be empty). The disjoint union of the hyperbolic pieces is called the hyperbolic part of $M$ and denoted by $M_\mathrm{hyp}$. Its complete finite volume hyperbolic metric is called $g_{\mathrm{hyp}}$. The hyperbolic part $(M_\mathrm{hyp}, g_{\mathrm{hyp}})$ is up to isometry canonically determined by $M$. We start with  the analogue of Theorem \ref{bcg}:

\begin{theo} \label{dim3} 
Let $M$ be a closed oriented $3$-manifold and let $(M_{\mathrm{hyp}},g_{\mathrm{hyp}})$ be its hyperbolic part with its hyperbolic metric.
Then 
$$\spherevol(M) = \Vol(M_{\mathrm{hyp}}, \frac{1}{3}g_{\mathrm{hyp}}).$$
\end{theo}

\begin{proof}[Outline of proof]
The proof results from a combination of ideas due to Besson-Courtois-Gallot, Souto and Pieroni \cite{BCG91,BCG95,Souto01, Pieroni19}. Let $M$ be a closed oriented $3$-manifold. By the Geometrization theorem, it is known that $M$ is the connected sum of irreducible manifolds $Y_1,...,Y_m$ such that each $Y_j$ is decomposed, after cutting along a disjoint collection of essential tori, into  hyperbolic pieces and  Seifert pieces (such pieces could be empty). Let $Y_1,...,Y_k$ ($k\in \{1,...,m\}$) be the summands with a non-empty union of hyperbolic JSJ components called $H_j\subset Y_j$. The manifold $H_j$ is not necessarily connected and if non-empty, $Y_j\setminus H_j$ is a graph manifold with boundary.  Let $g_{j,\mathrm{hyp}}$ be the complete finite volume metric on $H_j$. The disjoint union of $H_j$ is the hyperbolic part of $M$ and is denoted by $M_\mathrm{hyp}$.

In \cite[Theorem 3.1, Theorem 3.13]{Pieroni19}, for any small $\hat{\delta}>0$, a nonpositively curved metric $g_{j,\hat{\delta}}$ approximating the hyperbolic metric $g_{\mathrm{hyp}}$ is constructed on $Y_j$ for $j\in \{1,...,k\}$:  there is in particular a region $\mathcal{R}_j \subset Y_j$ such that $g_{j,\hat{\delta}}$ is a hyperbolic metric on the $\frac{1}{\hat{\delta}}$-neighborhood of $\mathcal{R}_j$ in $Y_j$ with respect to $g_{j,\hat{\delta}}$, and 
\begin{equation} \label{hjg}
\Vol(H_j,g_{j,\mathrm{hyp}}) - \Vol(\mathcal{R}_j, g_{j,\hat{\delta}}) \leq \hat{\delta}.
\end{equation}
Accordingly, a metric space $(X,\mathbf{d}_{\hat{\delta}})$ obtained by attaching $Y_1,...,Y_k$ at a common point $\mathbf{p}$ and collapsing the other summands $Y_{k+1},...,Y_m $ to $\mathbf{p}$ is defined in \cite[Section 6]{Pieroni19}. By a slight abuse of notations, we consider $Y_1,...,Y_k$ as subsets of $X$. Set
$$\Gamma:=\pi_1(X).$$
Let $\tilde{X}$ be the universal cover, endowed with the induced distance denoted by $\tilde{\mathbf{d}}_{\hat{\delta}}$ and let $o\in \tilde{X}$ be a reference point projecting to $\mathbf{p}\in X$. Except at countably many points of $\tilde{X}$ corresponding to the lifts of $\mathbf{p}$, $\tilde{\mathbf{d}}_{\hat{\delta}}$ is in fact given by the path metric of a smooth Riemannian metric, and moreover $(\tilde{X},\tilde{\mathbf{d}}_{\hat{\delta}})$ is a $\mathrm{CAT}(0)$ space.

 There is a variant of the barycenter map 
corresponding to $\Gamma$, $(\tilde{X}, \tilde{\mathbf{d}}_{\hat{\delta}})$ with properties completely similar to those stated in Section \ref{appendix b}:
$$\mathrm{Bar} : \mathbb{S}^+/\lambda_\Gamma(\Gamma) \to ({X},{\mathbf{d}}_{\hat{\delta}}).$$
In particular this map enjoys a  Jacobian bound of the type 
$$|\Jac \mathrm{Bar}| \leq 1 +\text{error}(\hat{\delta})$$ 
with respect to the metric $\frac{1}{3} \mathbf{d}_{\hat{\delta}}$, along any totally geodesic $3$-simplex $S$ in $\mathbb{S}^+/\lambda_\Gamma(\Gamma)$, and at any
point in $S$ which is sent inside $\bigcup_j \mathcal{R}_j \setminus \{\mathbf{p}\}$.
The error converges to $0$ as $\hat{\delta}$ goes to $0$.

Let $\Gamma_M:=\pi_1(M)$ and $h_M\in H_3(\Gamma_M;\mathbb{Z})$ the induced class. The natural homomorphism $\theta:\Gamma_M\to \Gamma$ induces a $1$-Lipschitz map 
$$\Theta :  {S^\infty}/\lambda_{\Gamma_M}(\Gamma_M) \to  {S^\infty}/\lambda_\Gamma(\Gamma),$$
see Subsection \ref{some distance non increasing maps}.

Let $C\in \mathscr{C}(h_M)$. As before, by Lemma \ref{approx polyh}, one can assume that $C$ is a polyhedral chain and each point of its support lifts to an $\ell^2$ function with finite support in $\Gamma_M$. Then $\Theta(\spt(C)) \subset \mathbb{S}^+/\lambda_\Gamma(\Gamma)$. An arbitrarily small perturbation of  $\Theta$, still denoted by $\Theta$ for simplicity, makes $\Theta_\sharp(C)$ a polyhedral chain in $\mathbb{S}^+/\lambda_\Gamma(\Gamma)$.
Set
$$\tilde{\mathrm{Bar}} := \mathrm{Bar}\circ \Theta.$$
By $\Gamma$-equivariance, 
$\tilde{\mathrm{Bar}}_\sharp(C) = \llbracket 1_M\rrbracket$. Note that $\tilde{\mathrm{Bar}}$ depends on $C$ because of the perturbation of $\Theta$ but this dependence will not play a role later.


As in the proof of Theorem \ref{bcg}, the Jacobian bound for the barycenter map  implies by the area formula of Subsection \ref{subsection:area} that for all $\epsilon>0$, whenever $\hat{\delta}$ is small enough,
\begin{align*}
\mathbf{M}(C) &\geq (\frac{1}{3})^{3/2} \sum_{j=1}^k \Vol(\mathcal{R}_j,g_{j,\hat{\delta}}) -\epsilon \\
& \geq  (\frac{1}{3})^{3/2}  \sum_{j=1}^k \Vol(H_j,g_{j,\mathrm{hyp}}) - 2\epsilon.
\end{align*}
Since $C\in \mathscr{C}(h_M)$ was arbitrary, by sending $\hat{\delta} \to 0$ we get 
$$\spherevol(M) \geq \sum_{j=1}^k \Vol(H_j,\frac{1}{3}g_{j,\mathrm{hyp}})) = \Vol(M_{\mathrm{hyp}}, \frac{1}{3}g_{\mathrm{hyp}}).$$

As for the reverse inequality
$$\spherevol(M) \leq  \Vol(M_{\mathrm{hyp}}, \frac{1}{3}g_{\mathrm{hyp}}),$$
 in view of \cite[Corollary 3.13]{BCG91} (see (\ref{bcgse}) in the proof of Theorem \ref{theorem:bcg}), one just needs to exhibit a metric on $M$ with volume entropy close to $2$ and volume close to that of the hyperbolic part. 
In \cite[Theorem 4.1]{Pieroni19}\footnote{Theorem 4.3 in \cite{Pieroni19} is not quite correct as stated,
but this is not a serious issue. The statement can be replaced by the following: for any $\epsilon>0$ and any two closed Riemannian manifolds $(Y',g')$, $(Y'',g'')$, there is a metric $g$ on the connected sum such that the volume entropies satisfy $h(g) \leq h(g') + h(g'') +\epsilon$ and $|\Vol(Y,g) - \Vol(Y',g')-\Vol(Y'',g'')| \leq \epsilon$.} the author constructs a metric $\mathbf{g}_{\hat{\delta}}$ on $M$ depending on $\hat{\delta}$, 
with the property that, as $\hat{\delta}\to 0$, 
\begin{itemize}
\item $\Vol(M,\mathbf{g}_{\hat{\delta}})$ converges to $\sum_{j=1}^k\Vol(H_j,g_{j,\mathrm{hyp}}) = \Vol(M_{\mathrm{hyp}}, g_{\mathrm{hyp}}) ,$
\item the volume entropy $h(\mathbf{g}_{\hat{\delta}})$ converges to $2$.
\end{itemize}
This finishes the outline.

\end{proof}

\subsection{Intrinsic uniqueness of spherical Plateau solutions}

In the main theorem of this section, we interpret the hyperbolic part of a closed oriented $3$-manifold as its unique spherical Plateau solution, up to intrinsic isomorphism. 

\begin{theo} \label{uniqueness2}
 Let $M$ be a closed oriented $3$-manifold, whose hyperbolic part is denoted by $(M_{\mathrm{hyp}}, g_{\mathrm{hyp}})$. Then any spherical Plateau solution $C_\infty$ for $M$ is intrinsically isomorphic to $(M_{\mathrm{hyp}},\frac{1}{{3}}g_{\mathrm{hyp}})$. 
\end{theo}

\begin{proof}[Outline of proof]
The proof is similar to that of Theorem \ref{uniqueness1}.
Let us give a rough idea of how it works.

Let $\{C_i \}\subset \mathscr{C}(h_M)$ be a sequence satisfying
\begin{equation}\label{limminimizing}
\lim_{i\to \infty} \mathbf{M}({C}_i) = \spherevol(M) = \Vol(M_\mathrm{hyp},\frac{1}{3}g_\mathrm{hyp}),
\end{equation}
where the second equality follows from Theorem \ref{dim3}.
Suppose that ${C}_i$ converges in the intrinsic flat topology to a spherical Plateau solution $C_\infty=(X_\infty,d_\infty,S_\infty)$. 

Let $(Y,\mathbf{d}_{\mathrm{hyp}})$ be the space obtained by taking $\bigsqcup_{j=1}^k (H_j,g_{j,{\mathrm{hyp}}})$ and identifying $k$ points $\{y_1,...,y_k\}$ which belong respectively to $H_1,...,H_k$. Denote by $\mathbf{p}$ this unique singular point of $Y$ (in general $Y$ is not path connected).

We will use the notations of the proof of Theorem \ref{dim3}. 
Take a sequence $\{\hat{\delta}_i\}$ converging to $0$ as $i\to \infty$. 
Consider the metric space $(X,\mathbf{d}_{\hat{\delta}_i})$ and its fundamental group  $\Gamma:=\pi_1(X).$
Recall that $(X,\mathbf{d}_{\hat{\delta}_i})$ is obtained by attaching summands $(Y_1,...,Y_k)$ at one point. By abuse of notations, we call that unique singular point of $(X,\mathbf{d}_{\hat{\delta}_i})$  by $\mathbf{p}$ too. The pointed  sequence $(X,\mathbf{d}_{\hat{\delta}_i},\mathbf{p})$ visibly ``converges'' to $(Y,\mathbf{d}_{\mathrm{hyp}},\mathbf{p})$ as $i\to \infty$.

Consider the  barycenter maps
$$\tilde{\mathrm{Bar}} :=\mathrm{Bar} \circ \Theta : \spt(C_i) \to (X,\mathbf{d}_{\hat{\delta}_i}),$$
see notations in the proof of Theorem \ref{dim3}.
Modulo some technical details, by following the proof of Theorem \ref{uniqueness1} step by step applied to $ \tilde{\mathrm{Bar}}$, we arrive at the following partial conclusion. Recall that 
$C_\infty=(X_\infty,d_\infty,S_\infty)$ is the spherical Plateau solution, limit of $C_i$.
There is a region $Z\subset X_\infty$ which is locally path connected, and if $L_{d_\infty}$  (resp. $L_{\mathrm{hyp}}$) is the path metric induced by $d_\infty$ (resp. $\mathbf{d}_{\mathrm{hyp}}$), 
$$(Z,L_{d_\infty}) \text{ is isometric to } (Y \setminus \{\mathbf{p}\}, \frac{1}{3}L_{\mathrm{hyp}})$$ 
via a map 
$$\tilde{\mathrm{Bar}}_{\infty}: Z\to Y \setminus \{\mathbf{p}\}$$
which is a ``limit'' of the barycenter maps $\mathrm{Bar}_i$. 
Moreover, 
$$(\tilde{\mathrm{Bar}}_{\infty})_\sharp S_\infty \llcorner Z = \llbracket 1_Y \rrbracket$$ where $\llbracket 1_Y \rrbracket$ is the natural current induced by the oriented finite volume Riemannian space $Y$.
We also have $\mathbf{M}(S_\infty\llcorner Z)=\mathbf{M}(C_\infty)$.

The final step is to show that $(X_\infty,L_{d_\infty})$ is in fact isometric to the $\frac{1}{3}L_{\mathrm{hyp}}$-completion of $(Y \setminus \{\mathbf{p}\},\frac{1}{3}L_{\mathrm{hyp}})$, which is isometric to the disjoint union
$$
\bigsqcup_{j=1}^k (H_j, \frac{1}{3}g_{j,\mathrm{hyp}}) = (M_\mathrm{hyp}, \frac{1}{3}g_\mathrm{hyp}).
$$
Essentially, what we want to rule out is that $(X_\infty,d_\infty)$ is isometric to a non-smooth space made of manifolds attached at one point, for instance 
$(Y,\frac{1}{3}L_{\mathbf{d}_{\mathrm{hyp}}})$.
Actually, the apparent issue with the point $\mathbf{p}$ is only an artefact of our definition of the barycenter map, and is not related to the geometry of spherical Plateau solutions. By playing with the important property that the set of Riemannian isometries of a finite union of finite volume hyperbolic $3$-manifolds is finite, and by choosing the attachment point $\mathbf{p}$ generically enough, we conclude that the only possibility is the desired statement: $(X_\infty,L_{d_\infty})$ is isometric as a length space to $(M_\mathrm{hyp}, \frac{1}{3}g_\mathrm{hyp})$ and the proof is completed.

\end{proof}

More details for the proof of Theorem \ref{uniqueness2} appeared in the earlier preprint \cite{Antoine22}.

By Theorems \ref{dim3} and \ref{uniqueness2}, the spherical volume behaves nicely under geometric decompositions of $3$-manifolds. To what extent this holds for general group homology classes is an interesting problem (see \cite{Gromov82} for such properties in the context of the simplicial volume):
\begin{question}[Additivity under connected sum]
Given two closed oriented manifolds $M_1$, $M_2$ of same dimension at least 3, if $M_1\sharp M_2$ denotes the connected sum, do we have
$$\spherevol(M_1\sharp M_2) = \spherevol(M_1)+\spherevol(M_1)?$$
Does any spherical Plateau solution for $M_1\sharp M_2$ decompose into the union of spherical Plateau solutions for $M_1$ and $M_2$?
\end{question}

Another intriguing problem suggested by Theorems \ref{dim3} and \ref{uniqueness2} is the following:
\begin{question} [RF and MCF]
What is the interplay between the spherical Plateau problem and the Ricci flow or the mean curvature flow?
\end{question}

  \section{Plateau Dehn fillings} 
  \label{section:uniqueness3}

\subsection{Preliminaries on $\mathrm{CAT}(0)$ Dehn fillings}  \label{fm prelim}
In \cite{FM10} Fujiwara and Manning constructed certain pseudomanifolds out of higher dimensional finite volume hyperbolic manifolds, which generalize $3$-dimensional Dehn fillings from a topological and group theoretic point of view. In dimensions larger than $3$, the spherical Plateau solutions associated to those pseudomanifolds play the role of hyperbolic Dehn fillings in dimension $3$. In this section we are interested in the asymptotic behavior of those ``Plateau Dehn fillings''.

First let us review some of the definitions and results in \cite{FM10}. Let $n\geq 3$. Consider a finite volume oriented hyperbolic $n$-manifold $(M,g_{\mathrm{hyp}})$ with disjoint toral cusps $E_1,...,E_m$ and no other ends. Here a toral cusp means an end of $M$ homeomorphic to the product of a torus with $(-\infty,0]$, such that the induced Riemannian metric on $\partial E$ is flat. Write $\bar{M}:= M\setminus \bigcup_{j=1}^m \mathring{E}_j$. Inside each component $\partial E_j$ of
$\partial \bar{M}$, choose an embedded totally geodesic torus $T_j$ of dimension $k_i$ where $k_j\in \{1,...,n-1\}$. $T_j$ is a leaf of a fibration $\partial E_j\to B_j$ with base an $(n-1-k_j)$-torus $B_j$ and leaves $k_j$-tori (note that in \cite{FM10}, contrarily to here, the ambient dimension is by convention $n+1$). One can form the topological space $M(T_1,...,T_m)$ by collapsing each leaf of these fibrations $\partial E_j\to B_j$ to points (see \cite[Definition 2.5]{FM10}). This space is a pseudomanifold which is smooth outside of the so-called filling cores $V_1,...V_m$ where $V_j$ is an $(n-1-k_j)$-torus.  $M(T_1,...,T_m)$ is not a manifold except when for all $j\in \{1,...,m\}$, $k_j=1$.  There is a natural map 
$$M\to M(T_1,...,T_m)$$
which induces a natural surjection 
\begin{equation}\label{surject fund}
\pi_1(M)\to \pi_1(M(T_1,...,T_m)).
\end{equation}

In dimension $3$ when each $k_j=1$, by classical works \cite{Thurston97,BH96} it is known that when each circle $T_j$ has length larger than $2\pi$, $M(T_1,...,T_m)$ is a hyperbolic manifold: this is the so-called $2\pi$ theorem.
The main theorem in \cite{FM10} generalizes the $2\pi$ theorem to higher dimensions.  It asserts that $M(T_1,...,T_m)$ admits a locally $\mathrm{CAT}(0)$ path metric $\mathbf{d}$ whenever none of the tori $T_j$ admit a closed geodesic of length at most $2\pi$, or equivalently when the injectivity radius of each $T_j$ with its intrinsic metric is strictly larger than $\pi$. Fujiwara and Manning conjectured in \cite[Conjecture 1.8, Question 1.9]{FM11} that as $\injrad(T_j)\to \infty$, the simplicial volume of $M(T_1,...,T_m)$ should converge to the simplicial volume of $M$ and approach that limit from below.

A useful property of the locally $\mathrm{CAT}(0)$ metric $\mathbf{d}$ constructed in \cite{FM10} is that it can be chosen to approximate the hyperbolic metric on $M$ on large sets when the injectivity radii of the $T_j$ are tending to infinity:
\begin{lemme}\label{bahoui} 
Fix a point $p\in M$. For any $R>0$, there is $i_R>0$ such that if
 \begin{equation} \label{i_R}
 \injrad(T_j)> i_R \quad \text{ for all $j\in \{1,...,m\}$},
 \end{equation}
then $\dim T_j < n-1$ for all $j\in \{1,...,m\}$ and one can choose the metric $\mathbf{d}$ on $M(T_1,...,T_m)$ so that there is a closed geodesic ball $\mathcal{B}_R$ of radius $R$ inside $(M(T_1,...,T_m),\mathbf{d})$ isometric to the closed geodesic ball of radius $R$ centered at $p$ inside $M$.
 \end{lemme}
 \begin{proof}
Note that when $i_R$ is large enough depending also on $M$, then (\ref{i_R}) necessarily implies that $T_j\neq \partial E_j$ i.e. $\dim T_j < n-1$ for all $j\in \{1,...,m\}$.
For any $\epsilon>0$, if $i_R$ is large enough and satisfies (\ref{i_R}),
then we can replace the toral cusps $E_j$ by new toral cusps $\tilde{E}_j \subset E_j$ which are far away inside the ends of $M$, so that $\Vol(M\setminus \bigcup_{j=1}^m {\tilde{E}}_j,g_\mathrm{hyp}) \geq \Vol(M,g_\mathrm{hyp}) -\epsilon$ and yet 
$$ \injrad(\tilde{T}_j)> \pi,$$
where $\tilde{T}_j$ is the totally geodesic torus in $\partial \tilde{E}_j$ corresponding to $T_j\subset \partial E_j$. Then we apply \cite[Proposition 2.8]{FM10} to the fillings obtained with $\tilde{T}_j \subset \partial \tilde{E}_j$ (which is homeomorphic to $M(T_1,...,T_m)$).

 \end{proof}
 
Outside of the filling cores, $\mathbf{d}$ is locally induced by a Riemannian metric of strictly negative curvature \cite[Theorem 2.7]{FM10}. The $n$-dimensional Hausdorff measure on $M(T_1,...,T_m)$ coincides with the Lebesgue measure outside of the filling cores, which have $0$ Hausdorff measure. We denote by $ \tilde{M}(T_1,...,T_m)$ the universal cover of ${M}(T_1,...,T_m)$, and by $\Xi $ the singular set in $ \tilde{M}(T_1,...,T_m)$, namely the union of the lifts of the filling cores in ${M}(T_1,...,T_m)$. Each component of $\Xi$ is a totally geodesic embedded Euclidean space \cite[Subsection 4.6]{FM10}. On $ \tilde{M}(T_1,...,T_m)$, the lift of $\mathbf{d}$ is $\mathrm{CAT}(0)$ and the group 
$$\Gamma:=\pi_1(M(T_1,...,T_m))$$
acts freely properly cocompactly by isometries.

In the sequel we will always assume that $\min_{j=1}^m \injrad(T_j) >2\pi$, so that the path metric $\mathbf{d}$ constructed in \cite{FM10} exists. 
An important fact is that one can define a barycenter map associated with $(\tilde{M}(T_1,...,T_m),\mathbf{d})$:
$$\mathrm{Bar} : \mathbb{S}^+/\lambda_\Gamma(\Gamma) \to (\tilde{M}(T_1,...,T_m),\mathbf{d})$$
and it has similar properties to the one reviewed in Section \ref{appendix b} (here the definitions of $ \mathbb{S}^+$ and $\mathrm{Bar}$ need to be adapted). There is a technical difficulty caused by the fact that
the distance functions on $(\tilde{M}(T_1,...,T_m),\mathbf{d})$ are potentially non-smooth on large sets. For instance the fact that Hessian bounds still make sense in a weak sense and the fact that a cycle in $ {S^\infty}/\lambda_\Gamma(\Gamma)$ can be perturbed to a cycle in $\mathbb{S}^+/\lambda_\Gamma(\Gamma)$ require some arguments. We will not discuss those issues but we mention the following useful papers \cite{FM10,KL21}.

\subsection{Asymptotic rigidity of Plateau Dehn fillings}

For $n\geq 3$, let $(M,g_{\mathrm{hyp}})$ be non-compact oriented hyperbolic $n$-manifold of finite volume, with only toral cusps $E_1,...,E_m$. 
Let $T_i\subset \partial E_i$ be an embedded totally geodesic $k_i$-dimensional torus. We will always suppose that the injectivity radii of $T_i$ are larger than $\pi$. By residual finiteness, any finite volume hyperbolic manifold is finitely covered by such a hyperbolic manifold. Denote  by $g'$ the following rescaling of the hyperbolic metric on $M$:
$$g' :=\frac{(n-1)^2}{4n} g_{\mathrm{hyp}}.$$

Let $M(T_1,...,T_m)$ be a $2\pi$-filling constructed by Fujiwara-Manning in \cite{FM10} endowed with the metric $\mathbf{d}$ satisfying Lemma \ref{bahoui}, for some positive dimensional tori $T_i$, as explained in the previous subsection. Denote by $\Gamma$ (resp. $h \in H_n(\Gamma;\mathbb{Z})$) the fundamental group (resp. the fundamental class) of $M(T_1,...,T_m)$. 
We will say that $h$ is the $2\pi$-filling homology class corresponding to $T_1,...,T_m$.

Recall that  intrinsic equivalence for integral currents spaces is defined in  Definition \ref{definition:intrinsic isomorphism}.
The following theorem establishes the asymptotic rigidity of Plateau Dehn fillings as $\injrad(T_i)\to \infty$, which in particular implies the spherical volume analogue of the conjecture of Fujiwara-Manning \cite[Conjecture 1.8, Question 1.9]{FM11}. This behavior is completely analogous to what happens to hyperbolic $3$-dimensional Dehn fillings.

\begin{theo} \label{dehn fillings}
Let $(M,g_{\mathrm{hyp}})$ be a non-compact finite volume oriented hyperbolic manifold with toral cusps, then the following holds. 
\begin{enumerate}
\item
 For any $2\pi$-filling homology class $h$,
$$\spherevol(h) <  \Vol(M,g').$$

\item Consider a sequence of families of tori $T^p_1,...,T^p_m$ such that 
$$\lim_{p\to \infty} \min_{i=1}^m \injrad(T^p_i) =\infty.$$
Let $C_{p,\infty}$ be any spherical Plateau solution for the $2\pi$-filling homology class $h_p$ corresponding to $T^p_1,...,T^p_m$. Then  
$$ \lim_{p\to \infty} \mathbf{M}(C_{p,\infty}) =\lim_{p\to \infty} \spherevol(h_p) = \Vol(M,g'),$$
and $C_{p,\infty}$ subsequentially converges in the intrinsic flat topology to an integral current space which is intrinsically isomorphic to $(M,g')$.
\end{enumerate}
\end{theo}

\begin{proof}[Outline of proof]
Let us start with the strict inequality (1). 
Recall that 
$$\Gamma:=\pi_1(M(T_1,...,T_m)).$$
Let $\tilde{M}$ be the universal cover of $M$ endowed with the hyperbolic metric  $g:=g_{\mathrm{hyp}}$.  
Set
$$S_2(\tilde{M},g) := \{u\in L^2(\tilde{M},g); \quad \|u\|_{L^2} =1\}$$
and denote by $\lambda_{(\tilde{M},g)}$ the natural $\pi_1(M)$-action on this sphere.
The fundamental group $\pi_1(M)$ naturally surjects onto $\Gamma$ by (\ref{surject fund}), and as in Subsections \ref{some distance non increasing maps},
there is a natural distance non-increasing map 
$$\Theta : S_2(\tilde{M},g) /\lambda_{(\tilde{M},g)}(\pi_1(M)) \to   {S^\infty}/\lambda_{\pi_1(M)}(\pi_1(M))\to  {S^\infty}/\lambda_\Gamma(\Gamma)$$
where the first map is obtained by averaging on translates of a fundamental domain.

An explicit sequence of admissible immersions from $M$ to $S_2(\tilde{M},g) /\lambda_{(\tilde{M},g)}(\pi_1(M))$ whose image has volume converging to $\Vol(M,\frac{(n-1)^2}{4n}g)$ is given by immersions defined as follows.
Let $\tilde{\varphi}: [0,\infty) \to (0,1)$ be a nondecreasing smooth function such that the function
$$\tilde{\rho}_y(.):= \tilde{\varphi}(\dist_{g}(y,.))$$
is smooth and coincides with $\dist_{g}(y,.)$ outside of the $1$-neighborhood of $y\in \tilde{M}$.
For each $c> n-1$, set 
$$\tilde{\alpha}_c :x\mapsto \tilde{\alpha}_{c,x} := e^{-\frac{c}{2}\tilde{\rho}_x(.)} \in L^2(\tilde{M},{g}).$$ By homogeneity the norm $\|\tilde{\alpha}_{c,x}\|_{L^2}$ does not depend on $x\in \tilde{M}$.
Define the corresponding immersion by 
$$\overline{\mathcal{P}}_c : (\tilde{M},g') \to S_2(\tilde{M},g)$$
$$ x\mapsto \overline{\mathcal{P}}_{c,x} = \frac{\tilde{\alpha}_{c,x}}{\|\tilde{\alpha}_{c,x}\|_{L^2}} . $$
This map is $\pi_1(M)$-equivariant and descends to a map from $M$ to $S_2(\tilde{M},g) /\lambda_{(\tilde{M},g)}(\pi_1(M))$. 
By well-known properties of hyperbolic spaces and their compactifications \cite[Subsection 2.6]{BCG95}, as $c\to n-1$, for any unit vector $v\in T_x\tilde{M}$,
$$\lim_{c\to n-1}\|d_x\overline{\mathcal{P}}_{c,x}(v)\|^2_{L^2}=  \lim_{c\to n-1} \frac{c^2}{4} \int_{\tilde{M}}  |d_x\tilde{\rho}_y(v)|^2 \overline{\mathcal{P}}^2_{c,x}(y) dvol_{g}(y) = \frac{(n-1)^2}{4n}.$$ 
The convergence is uniform on $T\tilde{M}$.
In particular, the pull-back metric converges to 
$$g':=  \frac{(n-1)^2}{4n}g$$ as $c\to n-1$ and
\begin{equation}\label{pluss}
\lim_{c\to n-1} \Vol(M,(\Theta \circ \overline{\mathcal{P}}_c)^*\mathbf{g}_\mathrm{Hil}) = \Vol(M,g').
\end{equation} 

Let $\theta:\pi_1(M) \to \Gamma$ be the natural surjective homomorphism and fix a fundamental domain $D\subset \tilde{M}$.
Consider the immersion 
$$\Theta \circ \overline{\mathcal{P}}_c : M\to   {S^\infty}/\lambda_\Gamma(\Gamma).$$
Unwinding the definitions, we get for $x\in \tilde{M}$, $\gamma\in \Gamma$:
$$\Theta \circ \overline{\mathcal{P}}_{c,x} (\gamma) = \big[\sum_{\tau\in \theta^{-1}(\gamma)} \int_{\tau.D} \frac{\tilde{\alpha}_{c,x}^2}{\|\tilde{\alpha}_{c,x}\|_{L^2}^2}   dvol_{g}\big]^{1/2}.$$
Form this expression, one checks that $\Theta \circ \overline{\mathcal{P}}_c$ has uniformly bounded second derivatives in $x$ as $c\to n-1$.

Let $h\in H_n(\Gamma;\mathbb{Z})$ be the $2\pi$-filling homology class corresponding to $T_1,...,T_m$.
The push-forward current $(\Theta \circ \overline{\mathcal{P}}_c)_\sharp \llbracket 1_M\rrbracket $ is in general not an element of $\mathscr{C}(h)$ since $M$ is noncompact. Nevertheless given any $\eta>0$, we can consider toral cusps $\tilde{E}_1,...,\tilde{E}_j$ respectively far enough inside $E_1,...,E_j$ and ``close'' the image tori $\Theta \circ \overline{\mathcal{P}}_c(\partial \tilde{E}_j)$ without adding much volume (by the cone construction in $ {S^\infty}$), to obtain a new admissible Lipschitz map 
$$\Psi : M(T_1,...,T_m) \to  {S^\infty}/\lambda_\Gamma(\Gamma)$$
such that $\Psi_\sharp(\llbracket 1_{M(T_1,...,T_m)}\rrbracket) \in \mathscr{C}(h)$ and 
\begin{equation}\label{goofy}
\mathbf{M}(\Psi_\sharp(\llbracket 1_{M(T_1,...,T_m)}\rrbracket)) \leq \Vol(M,(\Theta \circ \overline{\mathcal{P}}_c)^*\mathbf{g}_\mathrm{Hil}) +\eta.
\end{equation}
Taking $\eta\to 0$, (\ref{pluss}) and the above inequality already imply
$$\spherevol(h) \leq  \Vol(M,g')$$

Suppose towards a contradiction that $\spherevol(h) = \Vol(M,g')$. 
Then (\ref{goofy}) implies that if $\mathcal{D}_{c,x}$ is the differential of $\Theta \circ \overline{\mathcal{P}}_c:(M,g')\to   {S^\infty}/\lambda_\Gamma(\Gamma)$ at $x\in M$, we must have for any unit norm tangent vector $v\in T_xM$
\begin{equation} \label{cn1D}
\lim_{c\to n-1}\|\mathcal{D}_{c,x}(v)\| =1
\end{equation}
and the convergence is uniform on compact sets in $ M$ since the second derivatives of $\Theta \circ \overline{\mathcal{P}}_c$ are uniformly bounded.

For any small enough $l>0$, there is a smooth closed curve $a:[0,l]\to (M,g')$ of length $l$ parametrized by arclength with $a(0)=a(1)$, such that 
\begin{itemize}
\item $a:[0,l]\to M$ to represents an element in the non-empty kernel $\ker(\theta)$ of $\theta :\pi_1(M) \to \Gamma,$
\item at any point of the closed curve $a$, the geodesic curvature is equal to $1$.
\end{itemize}
From what we said above, the image of 
$$\Theta \circ \overline{\mathcal{P}}_c \circ a : [0,1] \to  {S^\infty}/\lambda_\Gamma(\Gamma) $$
is a loop contained in $ {S^\infty}/\lambda_\Gamma(\Gamma)$ which is homotopically trivial (hence it lifts isometrically to the sphere $ {S^\infty}$) and has length at most $2l$ for all $c$ close enough to $n-1$. Moreover the norm of the second derivative of $\Theta \circ \overline{\mathcal{P}}_c \circ a $ is uniformly bounded from above. By (\ref{cn1D}), the norm of the differential of $\Theta \circ \overline{\mathcal{P}}_c \circ a$ converges to $1$ uniformly as $c\to n-1$. 
But if $l$ is small enough, it is impossible to have an almost arclength parametrized closed curve from $[0,l] $ to a Hilbert unit sphere with second derivatives bounded independently of $l$. This contradiction finishes the proof of the strict inequality.

As for (2), our proof uses the barycenter map in the same fashion as the proofs of Theorems  \ref{theorem:bcg}, \ref{dim3}, \ref{uniqueness1} and \ref{uniqueness2}.
Let $(M(T^{p}_1,...,T^{p}_m), \mathbf{d}_p)$ be a sequence of $2\pi$-fillings where 
\begin{equation}\label{tends to infini}
\lim_{p\to \infty} \min_{k=1}^m \injrad(T^p_k) = \infty.
\end{equation}
For each $p$, let 
$$\Gamma_p:=\pi_1(M(T^p_1,...,T^p_m)),$$
 let $h_p\in H_n(\Gamma_p;\mathbb{Z})$ be the corresponding homology classes and let $C_{p,\infty}$ be a spherical Plateau solution for $h_p$. By Wenger's compactness theorem \cite[4.19]{SW11}, one can assume that $C_{p,\infty}$ converges in the intrinsic flat topology to some limit integral current space
$$W=(X_W,d_W,S_W).$$ 
As a spherical Plateau solution, $C_{p,\infty}$ is the intrinsic flat limit of a minimizing sequence $\{C_{p,i}\}_{i\geq 0} \subset \mathscr{C}(h_p)$. 


From now on, we will consider
$$\mathbf{d}'_p:= \frac{(n-1)^2}{4n} \mathbf{d}_p.$$
In the sequel, Jacobians, lengths and distances will be computed with respect to $\mathbf{d}'_p$ on $M(T^p_1,...,T^p_m)$.
As we briefly explained earlier, there is a well-defined barycenter map
$$\mathrm{Bar} : \mathbb{S}^+/\lambda_{\Gamma_p}(\Gamma_p) \to  M(T^p_1,...,T^p_m)$$
 associated with 
$\Gamma_p$, $(\tilde{M}(T^p_1,...,T^p_m),\mathbf{d}'_p)$ (here $\mathbb{S}^+$ depends on $\Gamma_p$.).
As before, using Lemma \ref{approx polyh}, all the $C_{p,i}$ can be assumed to be polyhedral chains without loss of generality.
For each $p$ and each $i$, one can show that
$$({\mathrm{Bar}}_p)_\sharp(C_{p,i}) = \llbracket 1_{M(T^p_1,...,T^p_m)} \rrbracket$$
where $\llbracket 1_{M(T^p_1,...,T^p_m)} \rrbracket$ is the integral current of $M(T^p_1,...,T^p_m)$ representing the fundamental class $[M(T^p_1,...,T^p_m)]\in H_n(M(T^p_1,...,T^p_m);\mathbb{Z})$.

Given $\epsilon>0$, let $R$ be such that if
\begin{equation} \label{>iR}
\min_{k=1}^m \injrad(T^p_k)> i_R,
\end{equation} then for the ball $\mathcal{B}_R \subset M(T^p_1,...,T^p_m)$ as in Lemma \ref{bahoui}, 
\begin{equation} \label{vol/2}
\Vol(\mathcal{B}_{R/2},g') > \Vol(M,g')-\epsilon/2
\end{equation}
where $g':=  \frac{(n-1)^2}{4n}g$. 
As in Theorems \ref{theorem:bcg} and \ref{dim3}, the Jacobian bound for the barycenter maps and the area formula imply that for all $R$ large,
$$\mathbf{M}(C_{p,i}) > \Vol(M,g') - \epsilon.$$
Combined with the upper bound shown in (1), this already proves that 
\begin{equation} \label{limite hp}
\lim_{p\to \infty} \spherevol(h_p) = \Vol_n(M,g').
\end{equation}

It remains to explain why the intrinsic flat limit $W=(X_W,d_W,S_W)$ of $\{C_{p,\infty}\}$ is intrinsically isomorphic to $(M,g')$.
The proof is based on arguments similar to those  of Theorem \ref{uniqueness1} and Theorem \ref{uniqueness2}. 
By Lemma \ref{bahoui}, the pseudomanifolds $(M(T^p_1,...,T^p_m),\mathbf{d}'_p)$ contain larger and larger geodesic balls $\mathcal{B}_R$ isometrically contained in $(M,g')$ and
$$M = \bigcup_{R>0} \mathcal{B}_{R}.$$
 Repeating the arguments of Theorem \ref{uniqueness1},
we conclude that there is a region $Z\subset X_W$ which is path connected, and if $L_{d_W}$ (resp. $L_{g'}$) denotes the path metric induced by $d_W$ (resp. $g'$),
$$(Z,L_{d_W}) \text{ is isometric to } (M,g')$$
via a ``limit barycenter map'' 
$${\mathrm{Bar}}_\infty : Z\to M$$
such that $({\mathrm{Bar}}_\infty)_\sharp S_W\llcorner Z = \llbracket 1_M\rrbracket$. Moreover by lower semicontinuity of the mass, $\mathbf{M}(W) \leq \Vol(M,g')$ so by the above isometry ${\mathrm{Bar}}_\infty $, necessarily $X_W = Z$. This finishes the proof.

\end{proof}


More details for the proof of Theorem \ref{dehn fillings} appeared in the earlier preprint \cite{Antoine22}.
  
Theorem \ref{dehn fillings} provides many examples of sequences of spherical Plateau solutions $\{C_{p,\infty}\}$ which ``accumulate'' towards a limit. Note that this accumulation occurs from below in the sense that the mass of each $C_{p,\infty}$ is strictly less than the mass of the limit. The set of simplicial volumes \cite{Gromov82} of closed oriented manifolds is well-ordered in dimension $2$ and $3$, but not  in higher dimensions, see \cite{HL21}.
There is also a conjecture (due to Harold Rosenberg?) stating that the set of areas of closed minimal surfaces in the round $3$-sphere is well-ordered. 
By Theorem \ref{theorem:bcg}, Theorem \ref{dim3} and \cite[Proposition 3.9]{BCG91}, the set of spherical volumes of closed oriented manifolds is well-ordered  in dimensions $2$ and $3$.
All these elements suggest the following question:
\begin{question}[Well-ordering]
Is the set of spherical volumes of closed oriented manifolds well-ordered in any dimension?
\end{question}

In dimensions at least $4$, it is possible that the Plateau Dehn fillings we constructed have supports which are non-smooth, but are smooth on large domains, due to the $\epsilon$-regularity theorem in \cite{ADLS18}. In fact, no example of singular spherical Plateau solution is known:
\begin{question}[Singularity]
Is there a closed oriented manifold with a singular spherical Plateau solution?
\end{question}

 The situation treated in Theorem \ref{dehn fillings} is very special due to the existence of a model hyperbolic manifold and the availability of barycenter maps. 
Considering that there are group theoretic versions of Dehn fillings developed in \cite{Osin07,GM08}, it would be desirable to investigate Plateau Dehn fillings beyond the case of cusped hyperbolic manifolds.
\begin{question} [Convergence phenomenon]
Are there other instances where ``convergence'' of a sequence of pairs $(\Gamma_i,h_i)$ implies convergence of the corresponding spherical volumes and spherical Plateau solutions?
\end{question}

Theorem \ref{dehn fillings} involves a non-compact hyperbolic manifold $M$, whose fundamental group $\pi_1(M)$ surjects onto the fundamental groups of Dehn fillings $M(T_1,...,T_m)$. Note that here, $M$ does not determine a nontrivial group homology class due to its non-compactness, but we can nevertheless interpret the spherical Plateau problems for the Dehn fillings $M(T_1,...,T_m)$ as spherical Plateau problems for $M$ with respect to orthogonal representations obtained by composing of the fundamental group surjections and the regular representations. Those remarks point to the following problem:
\begin{question} [Generalization]
Can the spherical Plateau problem be extended to non-compact manifolds and general orthogonal representations?
\end{question}

\bibliographystyle{alpha}
\bibliography{biblio_22_10_11}

\end{document}